%% file: main.tex
\documentclass[12pt,oneside]{amsart}   	
\usepackage{geometry}                		
\usepackage{graphicx}				
\usepackage{amssymb}
\usepackage{mathtools}

\usepackage[new]{old-arrows}

\usepackage[shortlabels]{enumitem}

\usepackage{amsthm,amssymb,amsmath,amsfonts,extarrows}
\usepackage[all,cmtip,2cell]{xy}

\usepackage[colorlinks=false]{hyperref}

\usepackage{xfrac,faktor}

\usepackage{multirow}

\usepackage[dvipsnames]{xcolor}

\usepackage{comment}
\usepackage{tikz-cd}

\newtheorem{manualtheoreminner}{Theorem}
\newenvironment{manualtheorem}[1]{%
  \renewcommand\themanualtheoreminner{#1}%
  \manualtheoreminner
}{\endmanualtheoreminner}

\newtheorem{theorem}{Theorem}[section]
\newtheorem{lemma}[theorem]{Lemma}
\newtheorem{cor}[theorem]{Corollary}
\newtheorem{conj}[theorem]{Conjecture}
\newtheorem{prop}[theorem]{Proposition}
\newtheorem{question}[theorem]{Question}
\theoremstyle{definition}
\newtheorem{defn}[theorem]{Definition}

\numberwithin{equation}{section}

\newcommand{\overbar}[1]{\mkern 1.5mu\overline{\mkern-1.5mu#1\mkern-1.5mu}\mkern 1.5mu}
\newcommand{\AAA}{\mathbb{A}}
\newcommand{\CC}{\mathbb{C}}

\newcommand{\GG}{\mathbb{G}}

\newcommand{\NN}{\mathbb{N}}

\newcommand{\PP}{\mathbb{P}}
\newcommand{\QQ}{\mathbb{Q}}
\newcommand{\RR}{\mathbb{R}}
\newcommand{\TT}{\mathbb{T}}
\newcommand{\VV}{\mathbb{V}}

\newcommand{\ZZ}{\mathbb{Z}}

\newcommand{\calA}{\mathcal{A}}
\newcommand{\calB}{\mathcal{B}}

\newcommand{\calF}{\mathcal{F}}

\newcommand{\calH}{\mathcal{H}}

\newcommand{\calK}{\mathcal{K}}
\newcommand{\calL}{\mathcal{L}}
\newcommand{\calM}{\mathcal{M}}
\newcommand{\calN}{\mathcal{N}}
\newcommand{\calO}{\mathcal{O}}

\newcommand{\calR}{\mathcal{R}}
\newcommand{\calU}{\mathcal{U}}
\newcommand{\calV}{\mathcal{V}}
\newcommand{\calW}{\mathcal{W}}
\newcommand{\calX}{\mathcal{X}}
\newcommand{\calY}{\mathcal{Y}}
\newcommand{\calZ}{\mathcal{Z}}
\newcommand{\gothA}{\mathfrak{A}}

\newcommand{\ZVHS}{\mathbb{Z}\mathrm{VHS}}
\newcommand{\RVHS}{\mathbb{R}\mathrm{VHS}}

\newcommand{\gothH}{\mathfrak{H}}

\newcommand{\dual}{\vee}

\newcommand{\Ka}{K^{\mathrm{alg}}}

\DeclareMathOperator{\Amp}{Amp}

\DeclareMathOperator{\diag}{diag}

\DeclareMathOperator{\divv}{div}
\DeclareMathOperator{\End}{End}

\DeclareMathOperator{\Gal}{Gal}
\DeclareMathOperator{\GL}{GL}

\DeclareMathOperator{\Hom}{Hom}

\DeclareMathOperator{\imag}{Im}
\DeclareMathOperator{\id}{id}

\DeclareMathOperator{\Lie}{Lie}

\DeclareMathOperator{\NS}{NS}
\DeclareMathOperator{\Null}{\mathfrak{N}}

\DeclareMathOperator{\Pic}{Pic}
\DeclareMathOperator{\cPic}{\mathcal{P}ic}
\DeclareMathOperator{\hPic}{\widehat{Pic}}
\DeclareMathOperator{\tPic}{\widetilde{Pic}}

\DeclareMathOperator{\rank}{rank}
\DeclareMathOperator{\rdim}{rel. dim}

\DeclareMathOperator{\Spec}{Spec}

\DeclareMathOperator{\spec}{sp}

\DeclareMathOperator{\tr}{tr}

\DeclareMathOperator{\vol}{vol}

\begin{document}

\title{Intersecting subvarieties of abelian schemes with group subschemes I}
\author{Tangli Ge}
\address{Department of Mathematics\\
Princeton University\\
Princeton, NJ 08540\\
U.S.A.}
\email{\url{tangli@princeton.edu}}


\input{abstract}
\maketitle

\setcounter{tocdepth}{1}

 \tableofcontents

\section{Introduction}\label{sec_intro}
\input{s1}


\section{Preliminaries on abelian schemes}\label{sec_abelianschemes}
\input{s2}

\section{Bi-algebraic geometry of \texorpdfstring{$\gothA_g$}{the universal abelian variety}}\label{sec_biagebraic_geometry}
\input{s3}

\section{Betti foliation and Betti form}\label{sec_betti}

\input{s4}

\section{Nondegeneracy}\label{sec_nondegeneracy}
\input{s5}

\section{Adelic line bundles on abelian schemes}\label{sec_adelic}
\input{s6}

\section{Homomorphism approximation}\label{sec_homomorphism_approx}
\input{s7}

\section{Proof of the main theorem}\label{sec_proof_main_thm}
\input{s8}

\section{Applications}\label{sec_app}
\input{s9}


\bibliographystyle{abbrv}
\bibliography{arithmetic}

\end{document}

%% file: abstract.tex
\begin{abstract}
In this paper, we establish the following family version of Habegger's bounded height theorem on abelian varieties \cite{Habegger_BHConAV}: a locally closed subvariety of an abelian scheme with Gao's $t^{\mathrm{th}}$ degeneracy locus~\cite{Gao_Betti} removed, intersected with all flat group subschemes of relative dimension at most $t$, gives a set of bounded total height. Our main tools include the Ax--Schanuel theorem, and intersection theory of adelic line bundles as developed by Yuan--Zhang~\cite{YuanZhang_ALB}. As two applications, we generalize Silverman's specialization theorem \cite{Silverman_Specialization} to a higher dimensional base, and establish a bounded height result towards Zhang's ICM Conjecture~\cite{Zhang_ICM98}. 
\end{abstract}

%% file: s1.tex
Let $S$ be a normal quasi-projective variety defined over a number field $K$ with function field $F$, and let $\pi:\calA\rightarrow S$ be an abelian scheme with generic fiber $A$.

\subsection{Scenario one}
 By the Lang--N\'eron theorem~\cite{LangNeron_rationalpoints}, the group $A(F)$ is finitely generated. Consider the specialization map for the Mordell--Weil group
\[
\spec_s:A(F)=\calA(S)\rightarrow \calA_s
\]defined by restricting a section to the fiber $\calA_s$ over any $s\in S$. A well-known result of Silverman \cite[Thm.~C]{Silverman_Specialization} states that when $S$ is a curve and $A$ has no constant part, the specialization maps are injective for all closed points of $S$, except for a set of bounded height. Naturally, this raises the question of whether analogous results hold for higher-dimensional bases.

 Let $\sigma_1,\cdots,\sigma_r\in A(F)$ be linearly independent generators for $A(F)_\QQ$. Consider the section $\underline{\sigma}:=(\sigma_1,\cdots,\sigma_r)$ embedded in the fibered power $\calA^r_S$. For $\underline{\lambda}:=(\lambda_1,\cdots,\lambda_r)\in \ZZ^r$, denote by $\calH_{\underline{\lambda}}$  the flat group subscheme of $\calA^r_S$ given as the kernel of 
\[
\calA^r_S\rightarrow \calA,\quad \underline{x}=(x_1,\cdots,x_r)\mapsto \underline{\lambda}\cdot\underline{x}=\lambda_1 x_1+\cdots +\lambda_r x_r.
\]Observe that a linear relation $\underline{\lambda}\cdot\underline{\sigma}(s)=0$ corresponds to an intersection $\underline{\sigma}\cap\calH_{\underline\lambda}$ over $s\in S$. In other words, nontrivial linear relations among $\sigma_1,\cdots,\sigma_r$ are encoded in the intersection $\underline{\sigma}\cap \bigcup_{\underline{\lambda}\neq 0}\calH_{\underline{\lambda}}$. In general, one can not expect the intersections to be rare. The following is our first result, which recovers Silverman's theorem in the curve case:


\begin{manualtheorem}{A}[Theorem~\ref{thm_spec_maximal_variation}]\label{thm_A}
    If all simple abelian subvarieties of $A_{F^{\mathrm{alg}}}$ have maximal variation (i.e., the period maps are generically finite) and dimension at least $\dim S$, then the set of closed points $s\in S$ where $\spec_s$ fails to be injective, or equivalently the set of closed points of the intersection $\underline{\sigma}\cap \bigcup_{\underline{\lambda}\neq 0}\calH_{\underline{\lambda}}$,
    is contained in the union of a strict Zariski closed subset and a set of bounded height of $\underline{\sigma}$. 
\end{manualtheorem}

We will see that the maximal variation assumption ensures that the subvariety $\underline{\sigma}$ is \emph{optimally nondegenerate}, while the dimension assumption restricts the intersection to being, at most, \emph{just likely}. Notably, the above theorem represents a specific instance of a broader phenomenon applicable to sufficiently nondegenerate subvarieties.

\subsection{Scenario two}
For this subsection, we assume $S=C$ is a curve. Let $\langle\cdot,\cdot \rangle$ be the canonical height pairing associated to a relatively ample line bundle $\calL$ on $\calA/C$. Take a finitely generated torsion free subgroup $\Lambda\subseteq \calA(C)$ with linearly independent generators $\sigma_1,\cdots, \sigma_r$. Define the following function
 \[
h_{\Lambda}(s):=\det \left( \langle \sigma_i(s),\sigma_j(s)\rangle_{i,j}\right)
 \]for any $s\in C(K^{\mathrm{alg}})$.
In his ICM notes~\cite{Zhang_ICM98}, S.~Zhang made the following conjecture:
\begin{conj}[S.~Zhang]\label{conj_zhang}
 Let $\calA\rightarrow C$ be an abelian scheme on a curve over a number field $K$ whose generic fiber is geometrically simple of dimension $\geq 2$. Let  $\Lambda\subseteq \calA(C)$ be a finitely generated torsion free subgroup. There is $\epsilon>0$ such that 
 \[\{s\in C(K^{\mathrm{alg}}) \mid h_\Lambda(s)<\epsilon\}\]
 is finite.
\end{conj}

As observed by Poonen, Conjecture~\ref{conj_zhang} is not true without the dimension assumption, since a general section of an elliptic surface can have infinitely many torsion points. Nevertheless, we are able to obtain the following bounded height result with no restriction on the dimension of abelian varieties in the non-constant case.

\begin{manualtheorem}{B}[Theorem~\ref{thm_towardszhang}]\label{thm_B}
     Let $\pi:\calA\rightarrow C$ be an abelian scheme on a curve $C$ defined over a number field $K$, and assume $\calA/C$ has no constant part. Let $\Lambda\subseteq \calA(C)$ be a finitely generated torsion free subgroup. There is $\epsilon>0$ such that 
    \[
    \{s\in C(K^{\mathrm{alg}})\mid h_{\Lambda}(s)<\epsilon\}
    \]is a set of bounded height. In particular, there are only finitely many points $s\in C(K^{\mathrm{alg}})$ of bounded degree with $h_{\Lambda}(s)<\epsilon$.
\end{manualtheorem}

For the proof, we shall reduce the question by linear algebra to study the intersection
of $\underline{\sigma}=(\sigma_1,\cdots,\sigma_r)\subseteq\calA_C^r$ with a certain \emph{$\epsilon$-height neighborhood} of $\bigcup_{\underline{\lambda}\neq 0}\calH_{\underline{\lambda}}$.

\subsection{Scenario three}
Now assume $S$ is a point so that $A$ is an abelian variety over $K$. Consider a subvariety $X\subseteq A$. For $t\in \NN$, let $A_{\leq t}$ denote the union of all group subschemes $H\subseteq A$ with dimension $\leq t$. One would like to understand when the closed points of 
\[
X\cap A_{\leq t}
\]are contained in the union of a strict Zariski closed subset and a set of bounded height in $X$. Guided by the general principle of unlikely intersection, we should remove those positive dimensional subvarieties $Y\subseteq X$ if $Y$ has codimension $<t$ inside a coset (i.e. translate of a group subvariety). Such $Y$'s are expected to result in unbounded height for dimension reasons, at least when $A$ is a power of elliptic curves as verified by Viada~\cite{Viada_optimal}. The union of all such $Y$'s is called the \emph{$t$-anomalous locus} of $X$, which is shown to be Zariski closed by R\'emond~\cite{Remond_IntersectionIII}. Denote the locally closed subvariety deprived of the $t$-anomalous locus by $X^{\circ,t}$. Fix a N\'eron--Tate height function $\hat h$ on $A$ and let the $\epsilon$-height neighborhood of $A_{\leq t}$ be
\[
    C(\epsilon,A_{\leq t}):=\{x+y\mid x\in A_{\leq t}(K^{\mathrm{alg}}), y\in A(K^{\mathrm{alg}}),\hat h(y)<\epsilon\}.\]
Habegger establishes the following bounded height theorem~\cite{Habegger_BHConAV} on abelian varieties:
\begin{theorem}[Habegger]\label{thm_habegger}
     Let $A$ be an abelian variety over a number field $K$ with a fixed N\'eron--Tate height and let $X$ be a subvariety of $A$.  There exists $\epsilon>0$ such that the closed points of the intersection $X^{\circ,t}\cap C(\epsilon,A_{\leq t})$ form a set of bounded height.
\end{theorem}

Our main result is a family version of Habegger's theorem, implying the three scenarios above, which we now describe. 

\subsection{Main result}
Let $S$ be a normal quasi-projective variety over a number field $K$ and let $\pi:\calA\rightarrow S$ be an abelian scheme, with its relative dimension denoted by $g$. 
For $t\in \NN$, let $\calA_{\leq t}$ denote the union of all \emph{flat} group subschemes $\calH\subseteq \calA$ with relative dimension $\leq t$. 
Consider an embedding of $\calA$ in a projective space. The induced naive height function $h$ on the closed points of $\calA$ is called a \emph{total height}. Fix a fiber-wise N\'eron--Tate height $\hat h$ on $\calA/S$. Define, for $\epsilon>0$, the $\epsilon$-height neighborhood of $\calA_{\leq t}$ as
\[
C(\epsilon,\calA_{\leq t}):=\{x+y\mid x\in \calA_{\leq t}(K^{\mathrm{alg}}), y\in \calA_{\pi(x)}(K^{\mathrm{alg}}), \hat h(y)<\epsilon\max\{1,h(x)\}\}.
\]We are interested in understanding for a subvariety $\calX\subseteq\calA$ when the following intersection
\[
\calX\cap C(\epsilon,\calA_{\leq t})
\]is sparse for sufficiently small $\epsilon$. Now it seems less transparent in this family setting on what is anomalous; for instance, a horizontal constant section could also lead to unbounded height. Nonetheless, the analogue of coset in this mixed setting is called \emph{weakly special subvariety}, defined in Definition~\ref{def_ws}, which is, roughly speaking, a translate of a (not necessarily dominant-over-$S$) group subscheme by a constant section. 

To present the result concisely, we adopt an alternative approach compared to the main body of the paper. 

\begin{defn}
For simplicity, assume that $\End(A)=\End(A_{F^{\mathrm{alg}}})$, and consider a subvariety $\calX$ of $\calA$ which is not contained in any strict group subscheme of $\calA$.\footnote{In particular, $\calX$ is dominant over $S$ by considering the (non-flat) group subscheme given as the union of $\calA|_{\pi(\calX)}$with the identity section.} 

For $t\in\ZZ$, the subvariety $\calX$ is called \emph{$t$-nondegenerate} if, for any abelian subscheme $\calB$ with relative dimension $g_\calB$, the composition 
\[\varphi_{/\calB}:\calA\rightarrow \calA/\calB \rightarrow \gothA_{g-g_\calB}\] where the first map is the quotient, and the second map is a mixed period map to the universal abelian variety, satisfies
     \[\dim\calX- \dim\varphi_{/\calB}(\calX)\leq\max \{0,g_\calB-t\}.\]
Moreover, define the \emph{degeneracy threshold} of $\calX$ to be
\[
\tau(\calX):=\min_{\calB}\left\{g_\calB-\left(\dim\calX-\dim\varphi_{/\calB}(\calX)\right)\right\}
\]where the minimum is taken over all abelian subschemes $\calB$ of $\calA$ with the property that $\dim\calX-\dim\varphi_{/\calB}(\calX)>0$.
It is immediate to check that $\calX$ is $\tau(\calX)$-nondegenerate but $(\tau(\calX)+1)$-degenerate. By taking $\calB=\calA$, we see that $\tau(\calX)\leq g-\dim\calX$.
\end{defn}

We remark that the fibers of $\varphi_{/\calB}$ are weakly special of relative dimension $g_{\calB}$ by Lemma~\ref{lem_weakly_special_char}. Indeed, the $t$-nondegeneracy condition above ensures that, the union of those positive dimensional subvarieties $\calY\subseteq\calX$ which have codimension $<t$ inside a weakly special subvariety, is not Zariski dense in $\calX$. Such $\calY$'s have the potential to lead to unbounded height when intersected with $\calA_{\leq t}$ for dimension reasons. We have the following main result, as a consequence of our main Theorem~\ref{thm_main} and the criterion~\ref{thm_criterion_for_Xdeg_equal_X}:

\begin{manualtheorem}{C}\label{thm_C}
    Let $S$ be a normal quasi-projective variety over a number field $K$, and let $\pi:\calA\rightarrow S$ be an abelian scheme equipped with a fixed fiber-wise N\'eron--Tate height $\hat h$ and a total height $h$. Let $\calX$ be a subvariety of $\calA$ with $\tau(\calX)\geq0$. There exist $\epsilon>0$ and a Zariski dense open subset $\calU\subseteq\calX$ such that the closed points of
    $\calU\cap C(\epsilon,\calA_{\leq \tau(\calX)})$ form a set of bounded total height.  
\end{manualtheorem}

In fact, we may take $\calU$ as the complement of the \emph{degeneracy locus} $\calX^{\mathrm{deg}}(t)$ for $t=\tau(\calX)$, first introduced and studied by Gao \cite[Def.~1.6]{Gao_Betti}, which is the union of those ``$t$-anomalous'' subvarieties $\calY$. For a precise statement, we refer to Therorem~\ref{thm_main}. Additionally, Theorem~\ref{thm_criterion_for_Xdeg_equal_X} establishes the connection between nondegeneracy and the degeneracy locus. When the base $S$ is a point, Theorem~\ref{thm_main} is exactly Habegger's Theorem~\ref{thm_habegger}. This work builds upon and extends his ideas to the relative setting.

\subsection{Previous work and Pink's conjecture}
The study of bounded height results for the intersection of subvarieties with subgroups traces back to 1999, with the work of Bombieri--Masser--Zannier~\cite{BMZ_99}.  They demonstrated that the intersection of a curve $C$ in the algebraic torus $\GG_m^r$ with all algebraic subgroups of codimension $\geq 1$, though not finite, is of bounded height---provided $C$ is not contained within a strict coset. Intersections where the dimensions of the involved varieties complement each other will hereafter be referred to as \emph{just likely}.
Their height upper bound is used in combination with Lehmer-type lower bounds in the same paper to further show that if the intersection is altered to an \emph{unlikely} situation, i.e., where
$C$ intersects algebraic subgroups of codimension $\geq 2$, then the result is finite. 

In subsequent work \cite{BMZ_Anomalous}, Bombieri--Masser--Zannier propose a general bounded height conjecture on $\GG_m^r$, suggesting that the bounded height analogue in the \emph{just likely} setting still holds if $C$ is replaced by a more general variety $X$, excluding its \emph{anomalous locus}. They demonstrated the Zariski-closed nature of the anomalous locus in the same work. Moreover, \cite{BMZ_plane08} showed that in the unlikely setting, the bounded height conjecture leads to finiteness, even in scenarios analogous to the curve case.

The bounded height conjecture was later resolved by Habegger~\cite{Habegger_BHCtori} in an innovative way. In parallel to the case of tori, related questions have been explored for (semi-)abelian varieties. Notably, in a separate paper~\cite{Habegger_BHConAV}, Habegger established the bounded height theorem (Theorem~\ref{thm_habegger}) for abelian varieties. K\"uhne~\cite{Kuhne_BHCsemiAV} extended this result further to cover general semiabelian varieties. 

Now consider the abelian scheme $\calA/S$ as before. For $d\in \NN$, denote by $\calA^{[>d]}$ the union of all codimension $>d$ group subvarieties of the fibers of $\calA\rightarrow S$. A conjecture of Pink in the preprint \cite[Conj.~6.1]{Pink_comb_unpublished},
as an implication of his general conjecture on mixed Shimura varieties in the same paper, predicts the following in the unlikely situation:
\begin{conj}[Pink]\label{conj_pink}
    Consider an abelian scheme $\calA\rightarrow S$ over a number field $K$ and an irreducible closed subvariety $\calX$ of dimension $d$ that is not contained in any strict group subscheme of $\calA$. Then $\calX \cap \calA^{[>d]}$ is not Zariski dense in $\calX$. 
\end{conj}

Conjecture~\ref{conj_pink} is a profound and widely open problem in general. It is known to imply various arithmetic results including Faltings' theorem~\cite{Faltings_DAonAV} in the special case where $S$ is a point; see Zilber~\cite[Prop.~3]{Zilber_exp} and Pink~\cite[Thm.~5.4]{Pink_comb_published}. Several known cases for the non-constant setting of $\calA/S$ are as follows:
\begin{enumerate}
    \item The relative Manin--Mumford conjecture, which focuses on torsion points. This conjecture was initially investigated by Masser--Zannier~\cite{MasserZannier_torsionEC0} and has recently been proven in full generality by Gao--Habegger~\cite{GH_RMM}.
    \item The case where $S$ is a curve over a number field, and the curve $\calX\subseteq\calA$ intersects with the union of \emph{flat} group subschemes of relative codimension at least $2$. This is a combined result of works by R\'emond~\cite{Remond_IntersectionII}, Habegger--Pila~\cite{HP_ominimality} and a series of papers by Barroero--Capuano~\cite{BC_PowersofEC,BC_productsofEC,BC_CurvesinAbelianScheme}. 
\end{enumerate}
In both cases above, establishing height upper bounds is crucial. The cited works of Masser--Zannier and Barroero--Capuano rely on Silverman's specialization theorem~\cite[Thm.~C]{Silverman_Specialization} as a key input for deriving the height bounds. While in the cited work of Gao--Habegger, the height bound comes from the height inequality of their joint work with Dimitrov~\cite{DimitrovGaoHabegger_UML}. The robust Pila--Zannier method~\cite{PilaZannier_MM} serves as a replacement for Lehmer-type height lower bounds to ensure finiteness. 

The main theorem in this paper studies the \emph{just likely} setting, focusing on the portion of group subschemes arising from the generic fiber of $\calA/S$ as in the above two cases. We will formulate a more general Conjecture~\ref{conj_height} at the conclusion of this paper, complementing Pink's Conjecture~\ref{conj_pink}.

\subsection{Outline of the strategy}
 We now provide an overview of the proof of the main results. %

In \S\ref{sec_abelianschemes}, we introduce the pullback of symmetric $\RR$-line bundles via $\RR$-homomorphisms between two abelian schemes, laying the groundwork for continuity arguments.



 In the first part (\S\S\ref{sec_biagebraic_geometry}-\ref{sec_nondegeneracy}) of the paper, we establish several geometric results using functional transcendence. We begin by defining Gao's degeneracy locus in the complex analytic setting. We include for completeness the proof of the algebraicity of the degeneracy locus~\cite[Thm.~1.8]{Gao_Betti} in a general abelian scheme $\calA/S$, following Gao's idea in his proof of the universal abelian variety case. We establish the finiteness result for \emph{types} of \emph{vertically optimal} subvarieties (Theorem~\ref{thm_finiteness_vertically_optimal}). The algebraicity (Theorem~\ref{thm_degeneracy_Zariski_closed}) follows immediately as a consequence of upper semi-continuity. Along the path, we derive Theorem~\ref{thm_criterion_for_Xdeg_equal_X} as a robust equivalent criterion for (non)degeneracy. 

Furthermore, we apply the Ax--Schanuel theorem to establish a necessary condition (Theorem~\ref{thm_volume_form_vanishing}) for a subvariety $\calX\subseteq\calA$ achieve vanishing of the volume form associated with the pullback of a Betti form by a surjective $\RR$-homomorphism $f:\calA\dashrightarrow \calB$ of abelian schemes. Specifically, if the $t^{\mathrm{th}}$ degeneracy locus of a subvariety $\calX$ is not Zariski dense, we deduce positivity for any $f$ of relative dimension $\leq t$. These positivity results are interpreted as the positivity of the self-intersection number of a \emph{geometric invariant adelic line bundle}. We emphasize that it would be crucial to allow $\RR$-homomorphisms and $\RR$-line bundles, and it seems necessary to use Ax--Schanuel type results. 

In the second part \S\S\ref{sec_adelic}-\ref{sec_proof_main_thm}, we prove the main theorem using intersection theory. While drawing inspiration from Habegger's approach, generalizing his ideas presents notable challenges. These
difficulties stem from addressing the subtleties and complexities involved in extending the methods beyond abelian varieties to this broader framework. To overcome one of these issues, we use the intersection theory of adelic line bundles over quasi-projective varieties, as developed by Yuan--Zhang~\cite{YuanZhang_ALB}. This allows us to avoid unnecessary and unnatural compactifications, which often arise when working on non-complete general abelian schemes. After recalling the invariant adelic line bundles on abelian schemes, we connect the Betti forms to the Monge--Amp\`ere measure.  This connection helps us relate the self-intersection of the invariant adelic line bundle in Corollary~\ref{cor_intersection_to_form}. 

For the proof, the starting observation is that flat group subschemes of relative dimension $\leq t$ are contracted by surjective homomorphisms $\varphi:\calA\rightarrow \calB$ of relative dimension $\leq t$. By Poincar\'e's complete reducibility, there are finitely many possible choices of $\calB$ up to isogeny. Consider one such $\calB$. A key result (Proposition~\ref{prop_compactness}), motivated by Habegger's~\cite[Lem.~2]{Habegger_BHConAV}, states that there exists a compact subspace $\calK(\calA,\calB)$ of the open locus of surjective $\RR$-homomorphisms in $\Hom(\calA,\calB)_\RR$. This subspace includes a set of representatives under the left action of $\End(\calB)_\RR$. Using this compact subspace, we derive uniform upper bounds \eqref{eqn_upperbound_intersection} and lower bounds \eqref{eqn_lowerbound_intersection} for intersection numbers on pullback adelic line bundles by continuity. These bounds on intersections, combined with Siu--Yuan's bigness theorem, gives a \emph{height upper bound} (Proposition~\ref{prop_upper_bound}) for the total height of a general point of $\calX$ under positivity assumption. On the other hand, the total height of a $\calB$-Null point, that is, a point that can be killed by a surjective homomorphism $\calA\rightarrow \calB$, has a lower bound (Proposition~\ref{prop_lower_bound_for_Bnull}) by height properties. When taking the intersection and choosing constants properly, these two height bounds compete each other and give the desired height bound (Theorem~\ref{thm_for_one_quotient}), over a Zariski dense open subset of $\calX$. 


To complete the proof of Theorem~\ref{thm_main}, we need to first show that the degeneracy locus, initially defined over $\CC$, is defined over the base number field. Using the positivity properties of the $t$-nondegenerate subvariety $\calX$, we apply Theorem~\ref{thm_for_one_quotient} to the finitely many possible $\calB$'s. A Noetherian induction argument then enlarges the Zariski open subset to include the complement of the $t^{\text{th}}$ degeneracy locus.

\subsection{Plan of the paper}
In \S\S\ref{sec_biagebraic_geometry}-\ref{sec_nondegeneracy}, we work within the complex analytic category. For the remaining sections, we use the scheme theoretic language.

\begin{itemize}
    \item \S\ref{sec_abelianschemes}: We collect general facts about abelian schemes and make some basic deductions. Pullbacks by $\RR$-homomorphisms are discussed. 
    \item \S\ref{sec_biagebraic_geometry}: We review the setup of bi-algebraic geometry for the universal abelian variety, following \cite{Gao_mixed}. Some notions are slightly rephrased for clarity. We also present the weak Ax–Schanuel theorem and a Finiteness \`a la Bogomolov–Ullmo result of Gao, both of which will be used later.
    \item \S\ref{sec_betti}: Betti forms are constructed carefully. Compared to the available literature, we are also interested in Betti forms associated to general nef line bundles. An integrability property of those Betti forms arising from pullbacks is given at the end.
    \item \S\ref{sec_nondegeneracy}: Gao's notion of degeneracy locus over $\CC$ is defined here. We include a proof of the Zariski closedness for the degeneracy locus on general abelian schemes. Along the way, we obtain the criterion for degeneracy. Finally, the Ax--Schanuel theorem is applied to derive the necessary condition for the vanishing of the volume measure attached to the pullback Betti form.
    \item \S\ref{sec_adelic}: We briefly introduce the language of adelic line bundles on quasi-projective varieties \cite{YuanZhang_ALB}. The properties of intersection theory and heights are discussed. 
    \item \S\ref{sec_homomorphism_approx}: We show the main arithmetic result  by intersection theory. 
    \item \S\ref{sec_proof_main_thm}: The degeneracy loci are showed to be defined over the base field, and the main theorem is proved using all previous results.
    \item \S\ref{sec_app}: The applications to Theorem~\ref{thm_A} and Theorem~\ref{thm_B} are discussed in detail. At the end, we state a conjecture as a potential strengthening of the main theorem.    
\end{itemize}

\subsection*{Acknowledgments}
The author would like to thank Ziyang Gao and Shou-Wu Zhang for constant discussions, for their encouragement and for the comments on the first draft; Dan Abramovich and Joseph Silverman for reading the first draft and for their suggestions; Philipp~Habegger for answering a related question and for his beautiful work~\cite{Habegger_BHConAV} which motivates this result; and Laura~DeMarco for explaining Silverman's proof to the author.



\addtocontents{toc}{\protect\setcounter{tocdepth}{0}}
\section*{\textbf{Notations and conventions}}
\begin{enumerate}[leftmargin=*]
    \item  We work in characteristic $0$ exclusively. Unless otherwise stated, $K$ denotes either a number field or $\CC$. An algebraic closure of $K$ is denoted by $\Ka$.
    \item Schemes are noetherian. Varieties are separated, geometrically irreducible and reduced schemes of finite type over the field. Subvarieties are Zariski closed unless otherwise specified such as locally closed subvarieties, which are open subsets of subvarieties. Images of varieties under a morphism are taken in the scheme-theoretic sense.
    \item An inclusion $X\subseteq Y$ is called \emph{strict} if $X\neq Y$.
    \item We follow the common definition of abelian schemes and group schemes. Abelian schemes are by definition flat over the base. Group schemes may have components that are not dominant over the base. Group subschemes (or subgroup schemes) are assumed to be Zariski closed throughout. Note that each component of \emph{flat} group subschemes of abelian schemes are dominant over the base variety by properness and flatness.

    \item Tensor products of (adelic) line bundles are written additively, e.g. $2\calL:=\calL^{\otimes 2}$. 
    \item Intersection products of adelic line bundles are written multiplicatively with a square bracket and subscript, e.g. $[\overbar L^d]_X$ means the $d$-th self intersection of $\overbar L$ on $X$.
    \item For an abelian group $V$ and a ring $R=\QQ \text{ or }\RR$, we write $V_{R}:=V\otimes_\ZZ R$. For a variety $X$ over $S$ and $T\rightarrow S$, we write $X_{T}:=X\times_S {T}$.
    \item The symbols $\calA,\calB$ are reserved for abelian schemes. The symbols $S,T$ are reserved for the base varieties, which are normal and quasi-projective over the field. 
    \item The symbols $\calX,\calY,\calU$ are reserved for (locally closed) analytic or algebraic subvarieties of abelian schemes. When $\calX$ is equipped with a natural projection map $\pi$, denote the relative dimension by $\rdim \calX:=\dim \calX-\dim\pi(\calX)$.
    \item Finite dimensional real vector spaces such as $\Hom(\calA,\calB)_\RR$ are equipped with the Euclidean topology and a fixed norm $|\cdot|$.
    \item For a complex manifold $Y$, the real (resp. holomorphic) tangent bundle is denoted by $TY$ (resp. $T'Y$).
\end{enumerate}

\addtocontents{toc}{\protect\setcounter{tocdepth}{1}}

%% file: s2.tex
In this section, we define basic terminology and fix notations in the language of schemes.
Let $S$ be a normal variety over a field $K$ of characteristic $0$ with generic point $\eta$. An \emph{abelian scheme} $\pi: \calA\rightarrow S$ is a group scheme which is smooth and proper with geometrically connected fibers. The zero section of $\calA$ is denoted by $e:S\rightarrow \calA$.  The relative dimension of $\calA/S$ is denoted by $g$. For any $l\in\ZZ$, let $[l]:\calA\rightarrow\calA$ be the multiplication by $l$. For $s\in S$, let $\calA_s$ be the fiber of $\calA$ over $s$.

\subsection{Picard group}\label{subsec_NSgp}
The Picard functor $\cPic(\calA/S)$ of $\calA/S$ is a functor from the category of schemes over $S$ to the category of groups defined by
\[
\cPic(\calA/S)(T):=\Pic(\calA_T)/\pi_T^*\Pic(T)
\]for any $S$-scheme $T$. We define the \emph{Picard group} for $\calA/S$ as
\[
\Pic(\calA/S):=\cPic(\calA/S)(S)=\Pic(\calA)/\pi^*\Pic(S).
\]A \emph{rigidified line bundle} on $\calA/S$ is a pair $(\calL,e^*\calL\cong \calO_S)$ with $\calL\in\Pic(\calA)$ such that $e^*\calL\in \Pic(S)$ is trivial and the fixed rigidification $e^*\calL\cong \calO_S$ is an isomorphism of line bundles. We usually leave out the rigidification when there is no ambiguity. An isomorphism of rigidified line bundles is an isomorphism of line bundles that is compatible with the rigidification. One checks immediately:
\[
\Pic(\calA/S)\cong \{\text{Rigidified line bundles on }\calA/S\}/\sim_{\mathrm{isom}}.
\]Thus, we shall regard an element in $\Pic(\calA/S)$ as a rigidified line bundle up to isomorphism. 

Two $\RR$-line bundles $\calL_1,\calL_2\in \Pic(\calA/S)_\RR$ are said to be \emph{numerically equivalent}, denoted as $\calL_1\equiv\calL_2$, if $\deg(\calL_1|_C)=\deg(\calL_2|_C)$ for any curve $C$ in a fiber of $\calA\rightarrow S$. If $\calL\in\Pic(\calA/S)_\RR$ is numerically equivalent to $\calO_\calA$, we say $\calL$ is numerically trivial. The \emph{N\'eron--Severi group} of $\calA/S$ is defined as
\[
\NS(\calA/S):=\Pic(\calA/S)/{\equiv}.
\]

A rigidified line bundle $\calL\in\Pic(\calA/S)$ is called \emph{symmetric} if  $[-1]^*\calL\cong \calL$.  Let $\Pic_0(\calA/S)$ be the subgroup of $\Pic(\calA/S)$ consisting of rigidified symmetric line bundles on $\calA$. Any $\QQ$-line bundle $\calL$ can be written as the sum of a symmetric line bundle $\frac{1}{2}(\calL+[-1]^*\calL)$ with a numerically trivial line bundle $\frac{1}{2}(\calL-[-1]^*\calL)$. So we may equivalently define $\NS(\calA/S)_\QQ$ as the quotient $\Pic_0(\calA/S)_\QQ/\equiv$. 


An $\RR$-line bundle on a projective variety is called ample if it is a positive linear combination of ample line bundles. $\RR$-ampleness is a numerical property; see \cite[Prop.~1.3.13]{Lazarsfeld_PositivityI}. An $\RR$-line bundle $\calL\in\Pic(\calA/S)_\RR$ is defined to be \emph{ample}\footnote{This is usually called ``relatively ample''. However, we believe ``ample'' is more suitable here when discussing relative classes, and is compatible when $S$ is a point.}, resp. \emph{nef}, if $\calL_s:=\calL|_{\calA_s}$ is ample, resp. nef, for any $s\in S$. Since ampleness and nefness are numerical, such definitions also make sense for a class in $\NS(\calA/S)_\RR$. We use the notation $\calL\geq0$ to mean $\calL$ is nef.


According to Raynaud \cite[ Cor.~VIII.7]{Raynaud_BookFA}, a line bundle $\calL\in\Pic(\calA/S)$ is ample if $\calL_s$ is ample for one point $s\in S$. Since $S$ is a normal variety, by a result of Grothendieck \cite[Thm.~XI.1.4]{Raynaud_BookFA}, given any ample line bundle $\calL_\eta\in\Pic(\calA_\eta)$, there is a symmetric line bundle $\calM\in\Pic_0(\calA/S)$ such that $\calM_\eta\equiv 2\calL_\eta$. Then $\calM$ is ample.  In particular, the abelian scheme $\calA/S$ is projective. 

The following proposition summarizes some basic numerical properties of rigidified line bundles on abelian schemes.

\begin{prop}\label{prop_numerical_property}
The following holds true:
\begin{enumerate}
    \item   The specialization map $\spec_s:\NS(\calA/S)\rightarrow \NS(\calA_s)$ is injective for any $s\in S$. Moreover, $\spec_{\eta,\QQ}:\NS(\calA/S)_\QQ\rightarrow\NS(\calA_\eta)_\QQ$ is an isomorphism.

    \item An $\RR$-line bundle $\calL\in\Pic(\calA/S)_\RR$ is ample, resp. nef, if $\calL_s$ is ample, resp. nef for one point $s\in S$.

    \item A numerically trivial $\RR$-line bundle $\calL\in\Pic(\calA/S)_\RR$ can be written as a linear combination of numerically trivial integral line bundles. In other words, there is a natural isomorphism
    \[(\Pic(\calA/S)/\equiv)\otimes_\ZZ\RR=:\NS(\calA/S)_\RR\xrightarrow{\cong} \Pic(\calA/S)_\RR/\equiv.\]
               
    \item An ample $\RR$-line bundle $\calL\in \Pic(\calA/S)_\RR$ can be written as a positive linear combination of ample integral line bundles. 
\end{enumerate}
\end{prop}
\begin{proof}
(1) Assume $\calL\in \Pic(\calA/S)$ with $\calL_s$ numerically trivial. Since $\calA/S$ is projective, there exists an ample $\calL_0\in\Pic(\calA/S)$. For any rational number $\epsilon>0$, the $\QQ$-line bundle $\calL+\epsilon\calL_0$ restricted to the fiber $\calA_s$ is ample. Then $\calL+\epsilon\calL_0$ is ample by the result of Raynaud. So $\calL$, being the limit of $\calL+\epsilon \calL_0$ as $\epsilon\rightarrow 0$, is nef. Similarly, $-\calL$ is nef. Thus $\calL$ is numerically trivial. The ``moreover'' part follows from the extension result of Grothendieck since any line bundle can be expressed as a difference of two ample line bundles.
  
(2) Assume $\calL\in\Pic(\calA/S)_\RR$ such that $\calL_s$ is ample for $s\in S$. Then the N\'eron--Severi class $[\calL_s]$ lies in the open convex ample cone $\Amp(\calA_s)$ of $\calA_s$. Also, $[\calL_s]$ is in the $\RR$-vector subspace $\spec_{s}(\NS(\calA/S))_\RR$. So
  \[
    \Amp(\calA_s)\cap\spec_{s}(\NS(\calA/S))_\RR
  \]is a nonempty open convex subset of $\spec_{s,\RR}(\NS(\calA/S)_\RR)$ containing $[\calL_s]$. By openness, there exists $[\calL_1],[\calL_2]\in \NS(\calA/S)$ such that $[\calL_s]=c_1[\calL_{1,s}]+c_2[\calL_{2,s}]$ for $c_1,c_2>0$. By (1), we have $[\calL]=c_1[\calL_1]+c_2[\calL_2]$. Note that $\calL_1,\calL_2$ is ample by Raynaud's result. Thus $\calL$ is ample.

  For nefness, one can use the same limit argument as in (1). 
  
(3) Assume $\calL=\sum_i c_i \calL_i$ is numerically trivial with $\calL_i\in\Pic(\calA/S)$ and $c_i\in \RR$. By (2), this is equivalent to the condition that $\calL_s$ is numerically trivial for a point $s\in S$. The condition that $\sum_i c_i \calL_{i,s}$ is numerically trivial is given by finitely many integer-coefficient linear homogeneous equations in $c_i$'s, determined by intersecting with a set of generators of the subgroup of $H_2(\calA_s,\ZZ)$ spanned by algebraic $1$-cycles. It follows from linear algebra that the solution space is generated by integer coefficient vectors. Hence we can formally rewrite $\sum_i c_i\calL_{i,s}=\sum_j c_j' \calL_{j,s}'$ where $\calL_{j}'$ is an integer linear combination of $\calL_{i}$ and $\calL_{j,s}'$ is numerically trivial. By (2) again, $\calL_{j}'$ is numerically trivial. Thus $\calL=\sum_j c_j'\calL_j'$ is a linear combination of numerically trivial integral line bundles.

(4) The basic observation is that for any real number $c$ and two line bundles $\calM_1,\calM_2$, letting $c',c''$ be rationals with $c'<c<c''$, there exists $0<t<1$ such that
  \[
    \calM_1+c\calM_2=t(\calM_1+c'\calM_2)+(1-t)(\calM_1+c''\calM_2).
  \] We first use the proof of (2) to express $\calL\equiv c_1\calL_1+c_2\calL_2$ with ample $\calL_1,\calL_2\in\Pic(\calA/S)$ and $c_1,c_2>0$. Now by (3), we can write $\calL-c_1\calL_1-c_2\calL_2$ as a linear combination of numerically trivial integral line bundles. The observation allows us to merge those with $c_1\calL_1$ and conclude. 

\vspace{-4mm}
\end{proof}



\subsection{Polarization and level structure}
The identity component $\cPic^0(\calA/S)$ of the Picard functor $\cPic(\calA/S)$ is representable by an abelian scheme $\calA^\dual/S$, known as the \emph{dual abelian scheme} to $\calA/S$. 
Given $\calL\in \Pic(\calA/S)$, define $\Lambda(\calL):\calA\rightarrow \cPic(\calA/S)$ as
\[
\mu^*\calL -p_1^*\calL -p_2^*\calL \in \cPic(\calA/S)(\calA)=\Pic(\calA\times_S\calA/\calA)
\]where $\mu,p_1,p_2:\calA\times_S\calA\rightarrow \calA$ are respectively addition, first and second projections. It can be easily checked that the image is in the identity component; that is, $\Lambda(\calL)$ is a homomorphism from $\calA$ to $\calA^\dual$.

A \emph{polarization} of an abelian scheme is defined as a homomorphism $\lambda:\calA\rightarrow \calA^\dual$ such that for any geometric point $\bar s$ of $S$, $\lambda_{\bar s}=\Lambda(\calL_{\bar s})$ for some ample line bundle $\calL_{\bar s}$ on $\calA_{\bar s}$. Any ample line bundle $\calL\in \Pic(\calA/S)$ defines a polarization $\Lambda(\calL)$. Conversely, $\lambda$ is not in general of the form $\Lambda(\calL)$. However, $2\lambda=\Lambda(\calL)$ for some ample $\calL\in \Pic(\calA/S)$ by \cite[Prop.~6.10]{Mumford_GIT}. The polarization is called \emph{principal} if it is an isomorphism. In general, the kernel of a polarization is a finite \'etale group subscheme over $S$, which over any geometric point, is isomorphic to 
\[
(\ZZ/d_1\ZZ\times \cdots\times \ZZ/d_g\ZZ)^2
\]for fixed positive integers $d_1|\cdots|d_g$. The diagonal matrix $D:=\diag(d_1,\dots,d_g)$ is called the \emph{type} of the polarization. Observe that polarization types are stable under base change. Because any abelian scheme over the normal base $S$ is projective, polarization always exists in our setting. 

A \emph{level-$n$-structure} on an abelian scheme in this article means a principal level-$n$-structure, i.e., an isomorphism $\alpha:(\ZZ/n\ZZ)^{2g}\xrightarrow{\sim}\calA[n]$, where $n\in \NN$. Its existence thus requires all $n$-torsion points of $\calA_\eta$ to be rational and the converse clearly holds as well. In particular, if $\calA/S$ has a level-$n$-structure, then so does any abelian subscheme $\calB\subseteq\calA$, or the quotient abelian scheme $\calA/\calB$.  We have the following convenient lemma for later use:

\begin{lemma}\label{lem_existlevelstr}
    For an abelian scheme $\calA/S$ and any $n\in\NN$, there exists a finite \'etale cover $S'\rightarrow S$ such that $\calA':=\calA\times_S S'$ has a level-$n$-structure.
\end{lemma}
\begin{proof}
    Indeed, let $T$ be an irreducible component of $\calA[n]$ which is not a section. Then $T$ is finite \'etale over $S$, of degree $>1$. Then $\calA_T$ has an $n$-torsion section given by the natural map $T\rightarrow \calA_T$. After repeating this finitely many times, we can get a finite \'etale cover $S'\rightarrow S$, such that all irreducible components of $\calA'[n]$ are sections.
\end{proof}


\subsection{Homomorphisms}\label{subsec_hom}
Let $\calB$ be another abelian scheme over $S$. A homomorphism of abelian schemes $\calA\rightarrow \calB$ restricts to a homomorphism of abelian varieties on the generic fiber $\calA_\eta\rightarrow \calB_\eta$. Conversely, a homomorphism $\calA_\eta\rightarrow \calB_\eta$ spreads out in a unique way to a homomorphism $\calA\rightarrow\calB$ by normality of $S$; cf. \cite[Prop.~I.2.7]{FC_DAV}. Thus there is a natural isomorphism of homomorphism groups
 \[
\Hom(\calA,\calB)\cong \Hom(\calA_\eta,\calB_\eta).
 \]In particular, $\End(\calA)\cong\End(\calA_\eta)$. 
 
 We will sometimes make the following innocuous extra assumption in the later discussion
\begin{equation}\label{eqn_all_endomorphism_appears}
   \End(\calA_\eta)=\End(\calA_{\bar\eta}).
\end{equation}This ensures that all abelian subvarieties of $\calA_{\bar \eta}$ are defined over $K(S)$.
Since $\End(\calA_{\bar\eta})$ is finitely generated, the assumption can be achieved by base changing $\calA_\eta$ to a finite extension of $K(S)$, which corresponds to taking a finite cover of $S$. Alternatively, one can apply a more precise result of Silverberg \cite[Thm.~2.4]{Silverberg_fod} to say it suffices to assume that there exists a level-$n$-structure on $\calA/S$ for some $n\geq3$. By Lem.~\ref{lem_existlevelstr} above, we may achieve \eqref{eqn_all_endomorphism_appears} by base changing to a finite \'etale cover of $S$.

Now suppose $B$ is an abelian subvariety of $\calA_\eta$. Then by Poincar\'e's complete reducibility \cite[Prop.~12.1]{Milne_AV}, there exists another abelian variety $B'$ and an isogeny $\calA_\eta\rightarrow B\times B'$. By composition with the projection and embedding, we get an endomorphism $\varphi_\eta:\calA_\eta\rightarrow\calA_\eta$ with $\varphi_\eta(\calA_\eta)=B$. Consider the corresponding homomorphism $\varphi\in \End(\calA)$. Note that, for any $\varphi\in \End(\calA)$, the image of $\varphi$ is an abelian subscheme of $\calA$; see \cite{ACM_Image} for a very general recent result. Thus $\varphi(\calA)$ is an abelian subscheme of $\calA$ with generic fiber $B$. As an immediate consequence, there is a bijection between
\[
\{\text{abelian subschemes of }\calA\}\leftrightarrow\{\text{abelian subvarieties of }\calA_\eta\}.
\]By complete reducibility of $\calA_\eta$, we have an isogeny decomposition
\[
\calA\sim \calA_1^{r_1}\times_S\cdots\times_S \calA_k^{r_k}
\]where $\calA_i$'s are \emph{simple} abelian subschemes of $\calA$ which are pairwise non-isogenous with $r_1,\cdots,r_k\in\NN\setminus\{0\}$. If we assume \eqref{eqn_all_endomorphism_appears}, then each $\calA_i$ is also geometrically simple. 

A particularly important portion to us is the fixed part (or constant part) of $\calA_\eta$. Denote the function field of $S$ by $F$. 
\begin{defn}
    Let $A$ be an abelian variety over $F$. The \emph{$F/K$-trace} of $A$ is a final object $(\tr_{F/K}(A),\tau)$ in the category of pairs $(B,f)$ where $B$ is an abelian variety over $K$ and $f:B_F\rightarrow A$ is a homomorphism. 
\end{defn}
Since varieties are assumed to be geometrically connected, the extension $F/K$ is primary. 
 The trace exists for the primary extension $F/K$, and in characteristic $0$, the homomorphism $\tau$ is an embedding; cf. \cite[\S6]{Conrad_Trace}. Denote by $\calA_0$ the abelian subscheme of $\calA$ corresponding to $\tau(\tr_{F/K}(\calA_\eta)_F)$. Note that $\calA_0\cong \tr_{F/K}(\calA_\eta)\times_K S$. The following definition is for later use:

\begin{defn}\label{defn_constant_multisection}
    A \emph{constant generic point} is defined to be a closed point $x\in \calA_\eta$ such that there exists a closed point $x'\in \tr_{F/K}(\calA_\eta)$, with base change still denoted by $x'\in \tr_{F/K}(\calA_\eta)_F$, and $N\in \NN\setminus\{0\}$ with $[N]x=\tau(x')$. If $x'=0$, such a constant generic point is called a \emph{torsion generic point}. A \emph{constant (resp. torsion) multi-section} is defined as the Zariski closure in $\calA$ of a constant (resp. torsion) generic point in $\calA_\eta$.
\end{defn}

 If the constant (resp. torsion) generic point is $F$-rational, then the corresponding constant (resp. torsion) multi-section is called a \emph{constant (resp. torsion) section}. 

 For later purpose, we are mostly interested in surjective homomorphisms. Suppose $\varphi:\calA\rightarrow\calB$ is a surjective homomorphism. Then $\calB$ is isogenous to $\calA_1^{r_1'}\times_S\cdots\times_S\calA_k^{r_k'}$ with $0\leq r_i'\leq r_i$. In particular, we get the following lemma which allows us to consider only finitely many targets eventually.

\begin{lemma}\label{lem_finite_target}
    There are only finitely many isogeny classes of abelian schemes $\calB$ over $S$ which admit a surjective homomorphism from $\calA$ to $\calB$.
\end{lemma}

For the rest of the section, we fix $\calB$ to be an abelian subscheme of $\calA$ from one of the isogeny classes in Lem.~\ref{lem_finite_target}. Let $\Hom(\calA,\calB)^\circ \subseteq \Hom(\calA,\calB)$ be the subset of surjective homomorphisms. Let $\Hom(\calA,\calB)^\circ_\QQ\subseteq\Hom(\calA,\calB)_\QQ$ be the subset of $f$ such that $N f$ is a surjective homomorphism for some integer $N\in\ZZ$. To use continuity arguments later, we also need a more general notion of surjective $\RR$-homomorphisms in $\Hom(\calA,\calB)_\RR$.

Note first that the composition of homomorphisms extends linearly to a pairing
\[
\begin{split}
  \Hom(\calB,\calA)_\RR   \times \Hom(\calA,\calB)_\RR&\longrightarrow \End(\calB)_\RR\\
  (\theta,f)&\longmapsto f\circ \theta
\end{split}
\]Define the subset of \emph{surjective $\RR$-homomorphisms} algebraically as
\[
\Hom(\calA,\calB)_\RR^\circ:=\{f\in\Hom(\calA,\calB)_\RR\mid \exists \theta\in \Hom(\calB,\calA)_\RR\text{ s.t. }f\circ \theta=\id_\calB\}.
\]This is consistent with the previous definition of surjective $\QQ$-homomorphisms due to complete reducibility.




We would like to study the union $\cup\ker\varphi$ for all $\varphi\in \Hom(\calA,\calB)^\circ$, which is equivalently the union of torsion translates of abelian subschemes of $\calA$ of fixed type isogenous to $\calA/\calB$. 
\begin{defn}
   A point $x\in \calA$ is called \emph{$\calB$-null}, if there exists a surjective homomorphism $\varphi:\calA\rightarrow \calB$ which sends $x$ onto the zero section. 
\end{defn}

Equip the finite dimensional vector space $\Hom(\calA,\calB)_\RR$ with the usual Euclidean topology. The following compactness result is one of the key ingredients to our main theorem, motivated by the analogous result of Habegger \cite[Lem.~2]{Habegger_BHConAV} on abelian varieties. We assume $K$ is a number field below for simplicity.
\begin{prop}\label{prop_compactness}
   Assume the hypothesis \eqref{eqn_all_endomorphism_appears}. There exists a compact subset $\calK(\calA,\calB)$ of $\Hom(\calA,\calB)^\circ_\RR$ with the property that for any $\calB$-null point $x\in \calA$, there is a $\QQ$-homomorphism $f\in \calK(\calA,\calB)$ and a nonzero $N\in\NN$ such that $Nf$ kills $x$.
\end{prop}
\begin{proof}
For any $ s\in S(\Ka)$, specialization gives an injection $\End(\calA)\hookrightarrow\End(\calA_s)$. Under the assumption \eqref{eqn_all_endomorphism_appears},  such specialization is indeed an isomorphism for``most'' of $s\in S(\Ka)$ by the main result of \cite{Masser_Specialization} (or by Hodge theory). We still denote such a point by $s\in S(\Ka)$. By our assumption, $\calB\subseteq\calA$. The isomorphism $\End(\calA)\cong\End(\calA_{s})$ also implies that 
$\End(\calB)\cong\End(\calB_{s})$, $\Hom(\calA,\calB)\cong\Hom(\calA_{s},\calB_{s})$ and $\Hom(\calB,\calA)\cong\Hom(\calB_s,\calA_s)$. By the purely algebraic definition of surjective $\RR$-homomorphisms, this means that we can identify
\[
\Hom(\calA,\calB)_\RR^\circ=\Hom(\calA_{s},\calB_{s})_\RR^\circ.
\]Now by \cite[Lem.~2]{Habegger_BHConAV} and its proof, there is a compact set $\calK(\calA_{s},\calB_{s})\subseteq \Hom(\calA_{ s},\calB_{ s})_\RR^\circ$ with the property that for any $f_s\in \Hom(\calA_{s},\calB_{ s})_\QQ^\circ$, there exists $\theta_s\in \End(\calB_{ s})_\QQ^\circ$ such that $\theta_s \circ f_s\in \calK(\calA_s,\calB_s)$. Let $\calK(\calA,\calB)$ be the corresponding set in $\Hom(\calA,\calB)_\RR^\circ$. If $x$ is $\calB$-null, there exists a surjective homomorphism $\varphi:\calA\rightarrow \calB$ which kills $x$. Let $\theta\in \End(\calB)_\QQ^\circ$ with $f:=\theta\circ \varphi\in \calK(\calA,\calB)$. Then for any nonzero $N\in \NN$ such that $N\theta$ is a homomorphism, $Nf=N\theta \circ \varphi$ kills $x$.
\end{proof}

\subsection{Pullback by \texorpdfstring{$\RR$-}{real }homomorphisms}
For any $\varphi\in \Hom(\calA,\calB)$, there is a pullback map $\varphi^*:\Pic(\calB/S)_\RR\rightarrow \Pic(\calA/S)_\RR$ which is a linear map preserving symmetry of line bundles. By abuse of notation, still write $\varphi^*:\Pic_0(\calB/S)_\RR\rightarrow \Pic_0(\calA/S)_\RR$ for the map restricted to symmetric line bundles. First, we extend it to $\QQ$-homomorphisms. Let $f\in \Hom(\calA,\calB)_\QQ$. Suppose $Nf\in \Hom(\calA,\calB)$ for nonzero $N\in \NN$. Define the pullback \[f^*:\Pic_0(\calB/S)_\RR\rightarrow \Pic_0(\calA/S)_\RR\] as $(Nf)^*/N^2$. This is well-defined since $[N]^*\calL=N^2\calL$ for a symmetric line bundle $\calL$ on an abelian scheme. Therefore we have 
\[
    \star_\QQ: \Hom(\calA,\calB)_\QQ\times \Pic_{0}(\calB/S)_\RR \rightarrow \Pic_{0}(\calA/S)_\RR
\]
sending $(f,\calL)$ to $f^*\calL$.

\begin{lemma}\label{lem_pullback_is_quadratic_form}
    The mapping $\star_\QQ$ is a $\QQ$-quadratic form on $\Hom(\calA,\calB)_\QQ$. Specifically, for any $f_1,f_2\in \Hom(\calA,\calB)_\QQ$, $\calL\in \Pic_{0}(\calB/S)_\RR$ and $n\in\QQ$, it is
    \begin{enumerate}
        \item homogeneous, i.e., $(n f_1)^*\calL=n^2 f_1^*\calL$, and
        \item quadratic, i.e., $(f_1+n f_2)^*\calL$ is a quadratic polynomial in $n$ with coefficients from $\Pic_0(\calA/S)$.
    \end{enumerate}
\end{lemma}
\begin{proof}
        It is clearly homogeneous of degree $2$. We show by induction that
    \begin{equation}\label{eqn_quadratic_pullback}
        (f_1+n f_2)^*\calL=(n^2-n)f_2^*\calL + n(f_1+f_2)^*\calL - (n-1)f_1^*\calL 
    \end{equation}for any $n\in\NN$. It holds trivially for $n=0,1$. Assume it holds for $n$. By the theorem of the cube \cite[Thm.~I.1.3]{FC_DAV}, there is linear equivalence:
   \[
         (f_1+nf_2+f_2)^*\calL= (f_1+nf_2)^*\calL + (nf_2+f_2)^*\calL + (f_1+f_2)^*\calL - f_1^*\calL - (nf_2)^*\calL - f_2^*\calL+ 0^*\calL.
   \]Note that $0^*\calL=\pi^*e^*\calL$ is trivial. Using induction hypothesis and symmetry of $\calL$, it is immediate to verify that
   \[
    (f_1+(n+1)f_2)^*\calL=(n^2+n)f_2^*\calL+(n+1)(f_1+f_2)^*\calL-nf_1^*\calL.
   \]Meanwhile, we have
   \[
    \begin{split}
        (f_1-nf_2)^*\calL=[(f_1+f_2)+(n+1)(-f_2)]^*\calL=(n^2+n)f_2^*\calL+(n+1)f_1^*\calL-n(f_1+f_2)^*\calL.
    \end{split}
   \]Thus \eqref{eqn_quadratic_pullback} holds for any $n\in\ZZ$, or $n\in\QQ$ by homogeneity.
\end{proof}

As a consequence, we have
\begin{prop}\label{prop_pullback_extends_to_real}
    There is a unique map
\begin{equation}\label{eqn_starR}
    \star: \Hom(\calA,\calB)_\RR\times \Pic_{0}(\calB/S)_\RR \rightarrow \Pic_{0}(\calA/S)_\RR
\end{equation}sending $(f,\calL)$ to $f^*\calL$ extending the pullback $\star_\QQ$, which is a real quadratic form in the first input and linear in the second input. Moreover, if $\calL\in \Pic_0(\calB/S)_\RR$ is nef, then $f^*\calL$ is nef for any $f\in \Hom(\calA,\calB)_\RR$. In particular, $\star$ descends to 
\begin{equation}\label{eqn_starR_for_NS}
    \star:\Hom(\calA,\calB)_\RR\times \NS(\calB/S)_\RR\rightarrow \NS(\calA/S)_\RR.
\end{equation}
\end{prop}
\begin{proof}
Let $\calL\in\Pic_0(\calB/S)_\RR$. Let $f_1,\cdots,f_r\in\Hom(\calA,\calB)_\QQ$ and $n_1,\cdots,n_r\in \QQ$. A quadratic form equivalently satisfies the following generalization of \eqref{eqn_quadratic_pullback}:
\begin{equation}\label{eqn_quadratic_pullback_general}
    (\sum_{i=1}^r n_i f_i)^*\calL =\frac{1}{2}\sum_{i,j=1}^r n_i n_j [(f_i+f_j)^*\calL-f_i^*\calL-f_j^*\calL].
\end{equation}
Then $\star$ can be uniquely defined by taking $f_1,\cdots,f_r$ as a $\QQ$-basis of the finite dimensional space $\Hom(\calA,\calB)_\QQ$ and allowing $n_i\in \RR$. 
One checks that \eqref{eqn_quadratic_pullback_general} implies $\star$ is a real quadratic form in the first input. The ``moreover'' part follows from continuity, since $f^*\calL$ is nef for any $f\in \Hom(\calA,\calB)_\QQ$.
\end{proof}

%% file: s3.tex
In this section, we present the framework of bi-algebraic geometry for the universal abelian variety and the associated functional transcendence result. The primary goal is to introduce the necessary notations for subsequent sections. For an elegant introduction to bi-algebraic geometry, we refer to the survey by Klingler--Ullmo--Yafaev \cite{KUY_survey}. For a detailed account of the Ax–Schanuel theorem in this context, we refer to Gao’s original paper \cite{Gao_mixed}.

\subsection{Universal abelian variety and its uniformization}\label{subsec_universalAV}
Let $g,n\geq 1$ be integers. Let $D:=\diag(d_1,\dots,d_g)$ denote a poloraization type with positive integers $d_1|\cdots|d_g$. Denote by $\AAA_{g,D,n}$ the moduli space of complex abelian varieties of dimension $g$ of polarization type $D$ with  level-$n$-structures, which admits a uniformization map from the Siegel upper half space
\[
\gothH_g:=\{Z\in M_{g}(\CC) \mid Z=Z^t, \imag(Z)>0\}
\]denoted by $u_{g,D,n}:\gothH_g\rightarrow \AAA_{g,D,n}.$ When $n\geq 3$, the moduli space $\AAA_{g,D,n}$ is a fine moduli space, that is, there exists a universal abelian variety $\pi_{g,D,n}:\gothA_{g,D,n}\rightarrow \AAA_{g,D,n}$; see \cite[Thm.~7.9]{Mumford_GIT} and its remark. We note that $\gothA_{g,D,n}$ and $\AAA_{g,D,n}$ are defined over $\QQ^{\mathrm{alg}}$.

The fiber product in the category of complex analytic varieties gives a diagram
\[
\begin{tikzcd}
    &\calA_{\gothH_g}:=\gothA_{g,D,n}\times_{\AAA_g}\gothH_g  \arrow[r]\arrow[d]  & \gothA_{g,D,n} \arrow[d]\\
    &\gothH_g \arrow[r,"u_g"] & \AAA_{g,D,n}
\end{tikzcd}.
\]The fiber of $\calA_{\gothH_g}$ over $Z\in \gothH_g$ is the complex abelian variety $\CC^g/(D\ZZ^g+Z\ZZ^g)$ with the obvious polarization and level structure. There is a natural uniformization of $\calA_{\gothH_g}$ given by $\CC^g\times\gothH_g \rightarrow \calA_{\gothH_g}$, such that, over $Z\in \gothH_g$, the map reduces to $\exp:\CC^g\rightarrow \CC^g/(D\ZZ^g+Z\ZZ^g)$. The composed map 
\[
u_{g,D,n,a}:\CC^g\times \gothH_g\rightarrow \calA_{\gothH_g}\rightarrow \gothA_{g,D,n}
\]is a uniformization of the universal abelian variety.


The polarization and level structure are not central to our discussion, as our primary focus is on dimension theory which is stable under finite coverings. While some flexibility in polarization is needed to handle abelian subschemes or quotients, the specific choice of polarization is not crucial. Therefore, for the rest of this paper, we \emph{fix $n=n_0\geq 3$ and omit the indices of polarization and level structure from the notation}.

The uniformization maps $u_g$ and $u_{g,a}$ are complex analytic and \emph{transcendental}. However, the targets have natural algebraic structures, as do the sources, which we now describe. For $\gothH_g$, the embedding
\[
\gothH_g\hookrightarrow \{Z\in M_g(\CC) \mid Z=Z^t\}\cong \CC^{g(g+1)/2}
\]realizes $\gothH_g$ as an unbounded semi-algebraic subset of the algebraic variety $\CC^{g(g+1)/2}$. A closed analytic subset $Y\subseteq \gothH_g$ is said to be \emph{irreducible algebraic} if $Y$ is a component of the intersection of an algebraic subvariety of $\CC^{g(g+1)/2}$ with $\gothH_g$. Similarly, we define the algebraic structure on $\CC^g\times \gothH_g$ by inheriting algebraic structure from $\CC^{g}\times \CC^{g(g+1)/2}$. One goal of bi-algebraic geometry is to understand the interplay between the two algebraic structures of the source and the target for such uniformization maps. The closed analytic subsets that are algebraic both on the source and the target, are of particular importance.
\begin{defn}
    A closed analytic subset $U\subseteq \gothH_g$ (resp. $\calU\subseteq \CC^g\times \gothH_g$) is called \emph{bi-algebraic} if $U$ (resp. $\calU$) is irreducible algebraic and $u_g(U)$ (resp. $u_{g,a}(\calU)$) is algebraic. In this case, we also say $u_g(U)\subseteq \AAA_g$ (resp. $u_{g,a}(\calU)\subseteq\gothA_g$) is bi-algebraic.
\end{defn}

\subsection{Weakly special subvarieties}\label{subsec_weaklyspecial}
 The relevant bi-algebraic subvarieties for $u_g$ (resp. $u_{g,a}$) are characterized in Ullmo--Yafaev \cite{UY_special} (resp. Gao \cite{Gao_APZ,Gao_towards}). Here, let us give a description of the bi-algebraic subvarieties for the universal abelian variety. 
 
 Let $\pi:\calA\rightarrow S$ be an abelian scheme over a normal complex algebraic variety $S$. We also assume $\calA$ has a level-$n_0$-structure and hence \eqref{eqn_all_endomorphism_appears}. Since $\CC(S)/\CC$ is a primary extension, the trace $(\tr_{\CC(S)/\CC}(\calA_\eta),\tau)$ exists and constant multi-sections as defined in Def.~\ref{defn_constant_multisection} make sense.

\begin{defn}\label{def_ws}
    Let $\calY\subseteq \calA$ be a subvariety.
    \begin{enumerate}
        \item We say $\calY$ is \emph{generically weakly special} (resp. \emph{generically special}) if $\calY=\calB+\sigma$ where $\calB\subseteq\calA$ is an abelian subscheme and $\sigma$ is a constant (resp. torsion) multi-section. 
        \item We say $\calY$ is  \emph{weakly special} (resp. \emph{special}), if $\calY$ is generically weakly special (resp. generically special) as a subvariety of the abelian scheme $\calA_{\pi(\calY)}$.
    \end{enumerate}
\end{defn}

Our definition of weakly special subvarieties differs slightly from the definition of weakly special subvarieties in the context of mixed Shimura varieties by Pink, but is in conformity with Klingler’s convention~\cite{Klingler_Hodge} in the context of variation of mixed Hodge structures; see Prop.~\ref{prop_bialgebraic_weaklyspecial}. Note that generically weakly special subvariety is called ``generically special subvariety of sg type'' in Gao~\cite[Def.~1.5]{Gao_Betti}.

\begin{lemma}\label{lem_ws_under_base_change}
       Let $\varphi:\calA'\rightarrow \calA$ be a base change corresponding to $S'\rightarrow S$. We have\begin{enumerate}
           \item If $\calY'$ is weakly special in $\calA'$, then $\varphi(\calY')$ is weakly special in $\calA$.
           \item If $\calY$ is weakly special in $\calA$, then any irreducible component of $\iota^{-1}(\calY)$ is weakly special in $\calA'$.
       \end{enumerate} 
   \end{lemma}
   \begin{proof}
       (2) is clear by definition. We prove (1). Suppose without loss of generality that $S'\rightarrow S$ is dominant and $\calY'=\calB'+\sigma'$ is generically weakly special in $\calA'$. Let $A,A',B',Y'$ be the respective generic fibers of $\calA,\calA',\calB',\calY'$. By Chow's theorem \cite[Cor.~3.21]{Conrad_Trace}, the abelian subvariety $B'\times_{\CC(S)}\CC(S)^{\mathrm{alg}}$ is the base change of an abelian subvariety $B_1\subseteq A_{\CC(S)^{\mathrm{alg}}}$. By assumption \eqref{eqn_all_endomorphism_appears}, $B_1=B_{\CC(S)^{\mathrm{alg}}}$ for some abelian variety $B\subseteq A$. Let $\calB$ be the abelian subscheme of $\calA$ with generic fiber $B$. By a similar argument, the fixed part of $A'$ is mapped to the fixed part of $A$ and the constant generic point of $\sigma'$ is mapped to a constant generic point, which gives a constant multi-section $\sigma$. Then $\varphi(\calY')=\calB+\sigma$ is weakly special.
   \end{proof}
   

Before going on, let us introduce the following terminology to avoid mentioning the specific choice of polarization and level structure:
\begin{defn}
    Let $g=\dim(\calA/S)$. A \emph{moduli sieve} (or \emph{mixed period map}) for $\calA/S$ is a morphism $\calA\rightarrow \gothA_g=\gothA_{g,D,n_0}$ induced from the period map $S\rightarrow \AAA_{g,D,n_0}$.
\end{defn}

The following lemma describes weakly special subvarieties as those that can be contracted by quotients followed by moduli sieves.
\begin{lemma}\label{lem_weakly_special_char}
    Assume $\calA/S$ has a level-$n_0$-structure. 
    \begin{enumerate}
        \item Let $\calB\subseteq \calA$ be an abelian subscheme with $\dim(\calB/S)=g'$. Let $\calA/\calB\rightarrow \gothA_{g-g'}$ be a moduli sieve of the quotient abelian scheme.  Then an irreducible component $\calY$ of any fiber of the composition $\calA\rightarrow \calA/\calB\rightarrow \gothA_{g-g'}$ is weakly special of relative dimension $g'$. 
        \item Conversely, if $\calY\subseteq\calA$ is a weakly special subvariety, then there exist abelian subschemes $\calB_1\subseteq \calA_1\subseteq\calA|_{\pi(\calY)}$ together with a positive integer $N$, such that $[N]\calY\subseteq\calA_1$ and $[N]\calY$ is an irreducible component of a fiber of the composition $\calA_1\rightarrow \calA_1/\calB_1\rightarrow \gothA_{g_1-g_1'}$, where $g_1,g_1'$ are the relative dimensions of $\calA_1,\calB_1$, respectively, and $\calA_1/\calB_1\rightarrow \gothA_{g_1-g_1'}$ is a moduli sieve.
    \end{enumerate}
\end{lemma}
\begin{proof}
    Indeed, we have the following diagram
    \[
    \begin{tikzcd}
        \calA \arrow[r,"q"]\arrow[rd]  & \calA/\calB \arrow[r,"\iota_a"]\arrow[d] & \gothA_{g-g'}\arrow[d]\\
          & S  \arrow[r,"\iota"] & \AAA_{g-g'}
    \end{tikzcd}.
    \]Note that for any $t\in\iota(S)$, an irreducible component  $T$ of $\iota^{-1}(t)\subseteq S$ gives rise to an isotrivial abelian scheme $\calA/\calB\times_S T\rightarrow  T$. Suppose $x\in \gothA_{g-g'}$ is a point over $t$. Then $\iota_a^{-1}(x)|_T$ is a constant section of the above isotrivial abelian scheme. By Poincar\'e's complete reducibility \eqref{subsec_hom}, there exists a homomorphism $\varphi:\calA/\calB\rightarrow \calA$ such that $q\varphi=\id_{\calA/\calB}$.
    Then $q^{-1}(\iota_a^{-1}(x)|_T)$ is the translate of $\calB\times_S T$ by the constant section $\varphi(\iota_a^{-1}(x)|_T)$. Therefore, any irreducible component of $q^{-1}(\iota_a^{-1}(x)|_T)$, which is also an irreducible component of $(\iota_a\circ q)^{-1}(x)$, is weakly special. We stress that it is in general \emph{not} generically weakly special.

    Conversely, if $\calY\subseteq\calA$ is weakly special, by definition, there exists a positive integer $N$ such that $[N]\calY$ is a translate of an abelian subscheme $\calB_1\subseteq \calA|_{\pi(\calY)}$ by a constant section $\sigma_c$. Let $\calA_1$ be the abelian scheme which is the sum of $\calB_1$ and the constant abelian subscheme in $\calA|_{\pi(\calY)}$ generated by $\sigma_c$. Then it is clear that $\calA_1,\calB_1,N$ satisfy the requirement.
\end{proof}

Back to the universal abelian variety, we have the following characterization by Gao:

\begin{prop}\label{prop_bialgebraic_weaklyspecial}
    The bi-algebraic subvarieties of $\gothA_g$ are exactly those weakly special subvarieties of $\gothA_g/\AAA_g$ which dominate bi-algebraic subvarieties of $\AAA_g$. Moreover, if $S\subseteq\AAA_g$ and $\calA=\gothA_{g,S}:=\gothA_g|_S$, then the weakly special subvarieties of $\calA/S$ come from restriction of bi-algebraic subvarieties, i.e., 
    \[
    \{\calY\subseteq\calA\mid \calY  \text{ is weakly special}\}=\{\calY\subseteq\calA\mid \calY \text{ is an irreducible component of }\calY^{\mathrm{biZar}}\cap\calA\}.
    \]
\end{prop}
\begin{proof}
   The first statement follows from \cite[Prop.~1.1]{Gao_APZ}. Note that weakly special subvariety has a different meaning in the loc. cit. and is equivalent to bi-algebraic subvariety by \cite[Cor.~8.3]{Gao_towards}.
   The second is \cite[Prop.~3.3]{Gao_APZ}.
\end{proof}


\subsection{Ax--Schanuel theorem}
One of the most powerful tools in studying the bi-algebraic geometry is a functional transcendence result called the \emph{Ax--Schanuel} theorem. Here, we recall the theorem in the context of universal abelian variety.

First we note that arbitrary intersection of algebraic subvarieties is algebraic. It follows formally that any intersection of bi-algebraic subvarieties is bi-algebraic. 
\begin{defn}
    Suppose $\calU\subseteq \CC^g\times \gothH_g$ (resp. $\calY\subseteq \gothA_g$) is a closed analytic irreducible subset. 
    \begin{enumerate}
        \item The \emph{Zariski closure} $\calU^{\mathrm{Zar}}\subseteq \CC^g\times \gothH_g$ (resp. $\calY^{\mathrm{Zar}}\subseteq \gothA_g$) is defined to be the smallest algebraic subvariety containing $\calU$ (resp. $\calY$).
        \item The \emph{bi-algebraic closure} $\calU^{\mathrm{biZar}}\subseteq \CC^g\times \gothH_g$ (resp. $\calY^{\mathrm{biZar}}\subseteq \gothA_g$) is defined to be the smallest bi-algebraic subvariety containing $\calU$ (resp. $\calY$).
    \end{enumerate} 
\end{defn}

The following is a weak version  \cite[Thm.~3.5]{Gao_mixed} of the Ax--Schanuel theorem which is sufficient for our purpose:

\begin{theorem}\label{thm_weakas}
    Let $u_{g,a}:\CC^g\times \gothH_g\rightarrow \gothA_g$ be the uniformization map, and let $\calU\subseteq \CC^g\times\gothH_g$ be a closed analytic irreducible subset. Then
    \[
    \dim \calU^{\mathrm{Zar}} -\dim \calU \geq \dim u_{g,a}(\calU)^{\mathrm{biZar}}-\dim u_{g,a}(\calU)^\mathrm{Zar}.
    \]
\end{theorem}


 Since $\pi_g(u_{g,a}(\calU)^{\mathrm{biZar}})$ contains $u_g(p(\calU^{\mathrm{Zar}}))$, where $p:\CC^g\times\gothH_g\rightarrow \gothH_g$ is the projection, we immediately have the following consequence:
\begin{cor}\label{cor_AxSchanuelVertical}
    Let $u_{g,a}:\CC^g\times \gothH_g\rightarrow \gothA_g$ be the uniformization map, and let $\calU\subseteq \CC^g\times\gothH_g$ be a closed analytic irreducible subset. Then
    \[
    \rdim \calU^{\mathrm{Zar}} -\dim \calU \geq \rdim u_{g,a}(\calU)^{\mathrm{biZar}}-\dim u_{g,a}(\calU)^\mathrm{Zar}.
    \]
\end{cor}



\subsection{Finiteness \`a la Bogomolov--Ullmo}
The \emph{weakly defect} of a locally closed subvariety $\calY\subseteq \gothA_g$ is defined as 
\[
\delta_{ws}(\calY):=\dim \calY^{\mathrm{biZar}} -\dim \calY.
\]Let $\calX\subseteq\gothA_g$ be a locally closed subvariety. A subvariety $\calY\subseteq\calX$ is called \emph{weakly optimal} in $\calX$ if, for any subvariety $\calY'\subseteq\calX$ containing $\calY$, one has $\delta_{ws}(\calY')>\delta_{ws}(\calY)$.

Later we need the following Finiteness \`a la Bogomolov--Ullmo\footnote{Analogous results are also termed ``Geometric Zilber--Pink'' in the literature.} result of Gao for weakly optimal subvarieties, which is itself an application of Ax--Schanuel, reformulated into our language as follows:
 
 \begin{theorem}\label{thm_GZP}
     Suppose $\calX\subseteq\gothA_g$ is a locally closed subvariety. There exists a finite set $\Sigma=\Sigma(\calX)$ of triples $(\calA_1,\calB_1,N)$ with  abelian schemes $\calB_1\subseteq\calA_1\subseteq\gothA_g|_{\pi(\calA_1)}$ and a positive integer $N$ with the following property. 
     
     For any weakly optimal subvariety $\calZ$ of $\calX$, there is a triple $(\calA_1,\calB_1,N)\in\Sigma(\calX)$ such that $[N]\calZ\subseteq\calA_1$ and $([N]\calZ)^{\mathrm{biZar}}$ is an irreducible component of a fiber of the composition
     \[
    \calA_1\rightarrow \calA_1/\calB_1\rightarrow \gothA_{g_1-g_1'}
     \] where the first map is quotient and the second is a moduli sieve with $g_1,g_1'$ being the relative dimensions of $\calA_1,\calB_1$. 
 \end{theorem}
 \begin{proof}
      The result \cite[Thm.~1.4]{Gao_mixed} (or \cite[Thm.~3.2.4]{Gao_HDR}) claims that there are finitely many pairs \[((Q,\calY^+),H)\in \Sigma'=\Sigma'(\calX)\]where $(Q,\calY^+)$ is a connected mixed Shimura subdatum and $H$ is a normal subgroup of $Q$ with semi-simple reductive part, such that any weakly optimal subvariety $\calZ$ is equal to $u_{g,a}(H(\RR)^+\tilde y)$ for one of the triples with some $\tilde y\in \calY^+$. By \cite[Prop.~5.6]{Gao_Betti}, a connected mixed Shimura subdatum $(Q,\calY^+)$ corresponds to a torsion translate of an abelian subscheme $\calA_1$ over a special subvariety of $\AAA_g$. Let $N\in \NN$ be an integer which kills the torsion.   By \S5.4 of loc. cit., $H$ gives rise to an abelian subscheme $\calB_1\subseteq\calA_1$ such that the quotient of $(Q,\calY^+)$ by $H$ corresponds to a moduli sieve of the quotient $\calA_1/\calB_1$. The fact that $\calZ^{\mathrm{biZar}}=u_{g,a}(H(\RR)^+\tilde y)$ means that $([N]\calZ)^{\mathrm{biZar}}$ is an irreducible component of a fiber of $\calA_1\rightarrow \calA_1/\calB_1\rightarrow \gothA_{g_1-g_1'}$. So we are done.
 \end{proof}

 

%% file: s4.tex
In this section, we define the notions of Betti foliation and Betti forms, and establish basic relations between them. A good supplement to our exposition is \cite[\S2]{CGHX_GBC}. Here, $S$ is a smooth complex quasi-projective variety viewed as a complex manifold, and $\pi:\calA\rightarrow S$ is an abelian scheme.

\subsection{Betti map and Betti foliation}
  Let $ \Delta\subseteq S$ be a simply connected analytic open neighborhood of $s\in S$. The local system $\VV_\ZZ:=(R^1\pi_*\ZZ)^\dual$ is trivial on $\Delta$ and there are trivializations $\VV_{\ZZ,\Delta}\cong \ZZ^{2g}\times \Delta$, and $\calV_{\Delta}\cong \CC^{2g}\times \Delta$ as holomorphic vector bundles. 
Denote the induced Hodge filtration on $\CC^{2g}\times \Delta$ by $\calF'^{\bullet}$. The composition map
 \[
 \RR^{2g}\times \Delta\rightarrow \CC^{2g}\times \Delta\rightarrow \CC^{2g}\times \Delta/\calF'^{0}
 \]is a real analytic diffeomorphism which restricts to a group isomorphism over any $s\in \Delta$. Taking the quotient by $\ZZ^{2g}\times \Delta$, it induces a real analytic diffeomorphism
 \begin{equation}\label{eqn_betti_map_real_analytic_diffeo}
      \TT^{2g}\times \Delta\cong \calA_\Delta
 \end{equation}
 which restricts to a real Lie group isomorphism of tori $\TT^{2g}\cong\calA_s$ for any $s\in \Delta$. The composition of its inverse with the projection 
 \[
 \beta_\Delta:\calA_\Delta\cong \TT^{2g}\times \Delta \rightarrow \TT^{2g}
 \]is called a \emph{Betti map} associated to the data $(\calA/S,\Delta)$. The Betti map $\beta_\Delta$ is uniquely determined up to an action of $\GL_{2g}(\ZZ)$ on the target. To further remove the ambiguity, we can use the isomorphism $\beta_\Delta|_{\calA_{s}}:\calA_s\rightarrow \TT^{2g}$ to define the composition
 \[
 \beta_{\Delta,s}:\calA_\Delta\xrightarrow{\beta_\Delta} \TT^{2g}\xrightarrow {(\beta_{\Delta}|_{\calA_s})^{-1}}  \calA_s
 \]which is called the \emph{Betti map} associated to the data $(\calA/S,\Delta,s)$. 
 
 The Betti map is real analytic but in general not holomorphic. On the other hand, any fiber of $\beta_\Delta$, being the image of 
 \[
    \{c\}\times \Delta\subseteq \CC^{2g}\times \Delta\rightarrow (\ZZ^{2g}\times \Delta)\backslash(\CC^{2g}\times \Delta)/\calF'^0
 \]for some $c\in \CC^{2g}$, is complex analytic, and gives rise to a local complex analytic foliation on $\calA_\Delta$. Local foliation patches together to a global foliation $\calF_{\text{Betti}}$ on $\calA$, which is called the \emph{Betti foliation}. A Betti leaf is a path-connected piece of the Betti foliation, and a path-connected component in the Betti foliation is called a maximal Betti leaf. We denote the maximal Betti leaf passing through $x$ by $\calF_{\text{Betti},x}$. For instance, a constant multi-section is a finite union of maximal Betti leaves. In general, a Betti leaf can be dense in $\calA$ under the analytic topology. The Betti foliation induces a splitting of the holomorphic tangent space
 \begin{equation}\label{eqn_splitting_tangent}
        T_x'\calA= T_x'\calF_{\text{Betti}} \oplus T_x'\calA_s
 \end{equation}
 at a point $x\in \calA_s$ with $s\in S$. Pullback via the identity section, the splitting gives
\begin{equation}\label{eqn_splitting_tangent_Lie_algebra}
    T'\calA|_S:=e^*T'\calA=T' S\oplus \Lie(\calA/S)
\end{equation}
where $\Lie(\calA/S)$ represents the \emph{Lie algebra} of $\calA/S$.

\subsection{An interlude to complex geometry}\label{subsec_complex_geometry}

Before defining Betti forms, let us quickly review a few concepts from complex geometry. Let $Y$ be a complex manifold and let $\omega$ be a real $C^\infty$ $(1, 1)$-form on $Y$. One way of thinking about the real $C^\infty$ (1,1)-form $\omega$ is as an alternating $\RR$-bilinear pairing $\omega_y:T_{y}Y \otimes_\RR T_{y}Y\rightarrow \RR$ on the real tangent bundle $T Y$ of $Y$, varying smoothly in $y\in Y$,  which satisfies $\omega_{y}(Jv,Jv')=\omega_y(v,v')$ for any $v,v'\in T_{y} Y$ with $J:T Y\rightarrow T Y$ the almost complex structure. If we view $\omega$ as a $C^\infty$ map of vector bundles $TY\rightarrow T^*Y$, then its kernel $\ker\omega$ is given by
\[
(\ker\omega)_y:=\ker\omega_y:=\{v\in T_y Y: \omega_y(v,v')=0\text{ for any }v'\in T_yY\}.
\]Clearly, $\ker \omega$ is stable under $J$. The kernel $\ker\omega$ is a vector bundle if and only if the rank of $\ker\omega$ is locally constant.  Assume moreover $\omega$ is semipositive, i.e., $\omega(v,Jv)\geq0$ for any $v\in TY$. Then 
\[
\ker\omega_y=\{v\in T_{y} Y: \omega(v,Jv)=0\}.
\]Indeed, for the less obvious direction with $\omega(v,Jv)=0$, notice that for $v'\in T_{y} Y$, we have
\[
\omega_y(v+kv',J(v+kv'))=k^2\omega_y(v',Jv')+2k\omega_y(v,Jv')\geq 0 \quad (\forall k\in \RR)
\]from which we derive $\omega_y(v,Jv')=0$. If $Z\subseteq Y$ is a complex submanifold and $z\in Z$, then the volume form $\omega_z^{\wedge \dim Z}|_Z\neq 0$ if and only if $\omega_z$ is positive on $T_z Z$, if and only if $ \ker\omega_z \cap T_z Z=\{0\}$.  

The next thing we recall is the first Chern class of line bundles. It is well-known that there is a first Chern class map for holomorphic line bundles
\[
c_1: \Pic(Y)= H^1(Y,\calO_Y^*)\rightarrow H^2(Y,\ZZ).
\]There is an analogous map for $C^\infty$ complex line bundles. The exact sequence 
\[
0\rightarrow \ZZ\rightarrow \calR_Y\xrightarrow{\exp} \calR_Y^*\rightarrow 0
\] where $\calR_Y$ (resp. $\calR_Y^*$) is the sheaf of $C^\infty$ complex functions (resp. nonzero $C^\infty$ complex functions), gives rise to an injective homomorphism 
\begin{equation}\label{eqn_chern_class_for_complex_line_bundle}
c_1': H^1(Y,\calR_Y^*)\hookrightarrow H^2(Y,\ZZ)
\end{equation}such that $c_1$ factors through $c_1'$; cf. \cite[pp.~140]{GH_principles}. Here $H^1(Y,\calR_Y^*)$ is canonically identified as the group of $C^\infty$ complex line bundles $L$ up to isomorphism using transition functions. Via the de Rham theory, we also regard the first Chern class $c_1'(L)$ as an integral class of real $2$-forms in $H^2(Y,\RR)\cong H^2_{\mathrm{dR}}(Y,\RR)$; a representative in the class can be given by $\frac{\sqrt{-1}}{2\pi}\Theta(D)$, called the \emph{Chern form} of $(L,D)$, where $\Theta(D)$ is the curvature form of a hermitian connection $D$ with respect to a hermitian metric on $L$; see \cite[\S\,V.9]{Demailly_Complex}.

\subsection{Betti form}\label{subsec_Bettiform}

Now let $\calL\in \Pic(\calA/S)$. For $s\in S$, there is a unique translation-invariant closed $(1,1)$-form $\omega_s$ on $\calA_s$ representing the Chern class $c_1(\calL_s)$ on $\calA_s$. The \emph{Betti form} associated to $\calL$ is the $(1,1)$-form  $\omega=\omega(\calL)$ on $\calA$ such that
\begin{enumerate}
    \item for any $s\in S$, the restriction $\omega|_{T\calA_s}=\omega_s$, and
    \item the form $\omega$ is rigidified along the Betti foliation, namely, $\omega|_{T\calF_{\text{Betti}}}=0$.
\end{enumerate}
This uniquely determines $\omega$ since it specifies the pairing $\omega_x:T_x\calA\times T_x\calA\rightarrow \RR$ at any point $x\in\calA$ by the splitting \eqref{eqn_splitting_tangent}. 

To see the continuous structure of Betti forms, we give a local construction. Let $\beta_{\Delta,s_0}:\calA_\Delta \rightarrow \calA_{s_0}$ be a Betti map and let $\omega_\Delta:=\beta_{\Delta,s_0}^*\omega_{s_0}$. By definition, $\omega_\Delta$ satisfies (2) as fibers of $\beta_{\Delta,s_0}$ are the Betti leaves.  To see (1), we regard $\calL_\Delta$ as a $C^\infty$ complex line bundle on $\calA_\Delta$, and consider the corresponding line bundle $\calL_{\Delta}'$ on $\TT^{2g}\times \Delta$ via the $C^\infty$ diffeomorphism \eqref{eqn_betti_map_real_analytic_diffeo}. Given any hermitian metric on $\calL'_\Delta$ and a hermitian connection $D$, the association of Chern forms on $\calL'_\Delta$ restricted to $\TT^{2g}\times \{s\}$ induces a continuous map
\[
\Delta\rightarrow  H^2_{\mathrm{dR}}(\TT^{2g},\RR)\cong H^2(\TT^{2g},\RR).
\] It factors through the lattice $H^2(\TT^{2g},\ZZ)$ and hence must be constant. This implies that for any $s\in \Delta$, the invariant form $\omega_s$ when viewed on $\TT^{2g}$ is independent of $s$. In particular, $\omega_{\Delta}|_{T\calA_s}=\omega_s$ for any $s\in \Delta$. Patching together by uniqueness, we get the desired Betti form. 

By definition, the Betti form of $\calL$ only depends on its class $[\calL]\in\NS(\calA/S)$. Linearity allows us to extend the definition to $\NS(\calA/S)_\RR$. It is not hard to establish the following properties of Betti forms.

\begin{prop}\label{prop_bettiform_property}
    The association of Betti forms 
    \[
    \omega:\NS(\calA/S)_\RR\rightarrow \Gamma(\Lambda^2 T^*\calA)
    \] 
    is an injective linear map. For $[\calL]\in \NS(\calA/S)_\RR$, we have
    \begin{enumerate}
        \item $\omega(\calL)$ is a  closed, real $C^\infty$ $(1,1)$-form,
        \item $\omega(\calL)$ is semipositive if $\calL$ is nef,
        \item $[N]^*\omega(\calL)=N^2\omega(\calL)$,
        \item $\ker\omega(\calL)$ is a $C^\infty$ complex vector bundle on $\calA$.
    \end{enumerate}
\end{prop}
\begin{proof}
    Injectivity of $\omega$ follows from the fact that the Chern class $c_1(\calL_s)=0$ if and only if $\calL_s$ is numerically trivial. All the listed properties are immediate consequences of the local description of $\omega(\calL)$.  
\end{proof}


Next, we discuss pullback of Betti forms under $\RR$-homomorphisms.
One concrete way to understand $\RR$-homomorphisms in the complex analytic setting is via Hodge theory. Let $\VV_\ZZ=\VV_{\ZZ,\calA}:=(R^1\pi_{*}\ZZ)^\dual$ be the $\ZVHS$ induced by $\pi:\calA\rightarrow S$. Let $\calV$ be the associated holomorphic vector bundle to $\VV_\CC$ with Hodge filtration $F^\bullet \calV$. Then we can identify 
\[\Lie(\calA/S)=\calV/F^0\calV\text{ and }\calA=\VV_\ZZ\backslash\calV/F^0\calV.\] The quotient map $\exp:\Lie(\calA/S)\rightarrow \calA$ is called the exponential map. 

Now an $\RR$-homomorphism $f:\calA\dashrightarrow \calB$ can be regarded as a map of $\RVHS$
\[
f_\RR:\VV_{\RR,\calA}\rightarrow \VV_{\RR,\calB}.
\]Note that $f$ is integral if and only if $f_\RR$ maps  $\VV_{\ZZ,\calA}$ to $V_{\ZZ,\calB}$. 
Taking into account the Hodge filtration, we have a commutative diagram
\[
\begin{tikzcd}
& \VV_{\RR,\calA}  \arrow[r,"\sim"] \arrow[d,"f_\RR"]    &   \Lie(\calA/S)   \arrow[d,"df"] \arrow[r,"\exp"]  & \calA \arrow[d,dashed,"f"]\\
&\VV_{\RR,\calB} \arrow[r,"\sim"]                &    \Lie(\calB/S) \arrow[r,"\exp"] & \calB 
\end{tikzcd}
\]This defines a functorial embedding
\begin{equation}\label{eqn_lie_alg_realization}
d:\Hom(\calA,\calB)_\RR\hookrightarrow \Hom(\Lie(\calA/S), \Lie(\calB/S))
\end{equation}into the group of global sections of the $\Hom$-bundle. If $f$ is surjective, then $df$ is surjective since a left inverse exists by functoriality. 

Alternatively, we can regard $df$ as the essential component of the differential map of $f$ on tangent spaces over the identity sections. Namely, under the decomposition \eqref{eqn_splitting_tangent_Lie_algebra}, the differential of $f$ is identity on $T'S$ and $df$ on the Lie algebras.

Define the pullback of a Betti form $\omega(\calL)$ for $[\calL]\in \NS(\calB/S)_\RR$ by $f\in\Hom(\calA,\calB)_\RR$ as 
\[
f^*\omega(\calL):=\omega(f^*\calL)
\]with $f^*\calL$ defined in \eqref{eqn_starR_for_NS}. This is consistent with the usual pullback of differential forms when $f$ is integral. By continuity we have $(df)^*\exp^*\omega(\calL)=\exp^*f^*\omega(\calL)$.

\begin{defn}
    An \emph{analytic subgroup} $\calH/S$ of $\calA/S$ is defined as the image of a flat $C^{\infty}$ complex vector subbundle $\calW\subseteq \Lie(\calA/S)$ on $S$ under $\exp:\Lie(\calA/S)\rightarrow \calA$.
\end{defn}

\begin{prop}\label{prop_integrable_via_analytic_subgroup}
    Let $[\calL]\in \NS(\calB/S)_\RR$ be ample, let $f\in \Hom(\calA,\calB)_\RR^\circ$. There exists an analytic subgroup $\calH$ such that for any $x\in \calA$, the  complex manifold $\calH_x$ obtained from translating $\calH$ by $\calF_{\text{Betti},x}$ is an integral manifold of $\ker f^*\omega(\calL)$ at $x$. Indeed, $\calH$ is given by $\exp(\ker(df))$, where $df:\Lie(\calA/S)\rightarrow \Lie(\calB/S)$ is the lift of $f$ by \eqref{eqn_lie_alg_realization}.
\end{prop}
\begin{proof}
Let $\calW:=\ker(df)$. Then $\calW$ is a $C^\infty$ complex vector bundle since $df$ has constant rank by surjectivity. Moreover, $\calW$ is flat by the identification $\calW\cong \ker(f_\RR)$ forgetting the complex structure, as $f_\RR$ is a map of local systems. We check that the analytic subgroup $\calH:=\exp\calW$ has the required property. Because of the translation-invariance of Betti form, it suffices to compare the tangent space of $\calH$ and $\ker f^*\omega(\calL)$ over the zero section of $\calA$.

Note first $\ker\omega(\calL)|_S=T'S$ on the identity section of $\calB$ by ampleness. Then 
$\ker f^*\omega(\calL)|_S$ consists of the tangent vectors which are mapped to $T'S$ by the differential of $f$. But the differential of $f$ is identity on $T'S$ and $df$ on the Lie algebras under \eqref{eqn_splitting_tangent_Lie_algebra}. Hence
\[
\ker f^*\omega(\calL)|_S=T'S\oplus \calW
\]is identical to the tangent space of $\calW$ or $\calH$ on the zero section, and we are done.
\end{proof}

%% file: s5.tex
We continue the discussion within the complex analytic category.
Let $\calA, \calB$ be abelian schemes over a normal complex quasi-projective variety $S$ of relative dimension $g,g'$, and let $f\in \Hom(\calA,\calB)^\circ_\RR$. 
Denote by $\omega:=\omega(f^*\calL_\calB)$ the Betti form of the pullback of a fixed ample line bundle $\calL_\calB\in\Pic_0(\calB/S)$. Let $\calX\subseteq\calA$ be a locally closed algebraic subvariety of dimension $d$. The wedge product $\omega^{\wedge d}|_\calX$ is a volume form on $\calX$. Since $\omega$ is semipositive, the volume form induces a nonnegative measure.


The main objective of this section is to provide a geometric criterion,  utilizing the Ax--Schanuel theorem, for determining when this measure is nontrivial, i.e., when the following nondegeneracy condition holds on $\calX$:
\begin{equation}
    \omega^{\wedge d}|_\calX \not\equiv 0.
\end{equation}

For simplicity, we assume $\calA$ has a level-$n_0$-structure and hence \eqref{eqn_all_endomorphism_appears}.

\subsection{Degeneracy loci}
We defined earlier in Def.~\ref{def_ws} the notion of a weakly special (resp. special) subvariety of $\calA/S$. Arbitrary intersection of weakly special (resp. special) subvarieties remains weakly special (resp. special). 

\begin{defn}
  Let $\calY\subseteq \calA$ be a locally closed subvariety.
  \begin{enumerate}
      \item The \emph{weakly special closure} $\langle \calY \rangle$ (resp. \emph{special closure}) of $\calY$ is defined as the smallest weakly special (resp. special) subvariety containing $\calY$. 
      \item The \emph{vertical defect} $\delta_v(\calY)$ is defined as the difference of the relative dimension of its weakly special closure $\langle\calY\rangle$ with the dimension of $\calY$. That is,
\[\delta_v(\calY):=\rdim \langle\calY\rangle -\dim\calY.\]
  \end{enumerate} 
\end{defn}

We have the following lemma:
\begin{lemma}\label{lem_relative_dim_of_ws_closure_under_moduli_sieve}
    Let $\iota_a:\calA\rightarrow \gothA_g$ be a moduli sieve. For a subvariety $\calY\subseteq \calA$,  we have
    \[
    \rdim \langle \calY\rangle =\rdim\langle \iota(\calY)\rangle.
    \]
\end{lemma}
\begin{proof}
  The pullback of $\langle\iota(\calY)\rangle$ is weakly special and contains $\calY$. So 
  \[
    \rdim \langle \calY\rangle \leq \rdim\langle \iota(\calY)\rangle.
    \]The reverse inequality follows from Lem.~\ref{lem_ws_under_base_change}.
\end{proof}

The following definition is introduced by Gao \cite{Gao_Betti} to study Betti maps.
\begin{defn}\label{def_degeneracy_over_complex}
  For $t\in\ZZ$, the \emph{$t^{\text{th}}$ degeneracy locus} of a locally closed subvariety $\calX\subseteq\calA$, denoted by $\calX^{\mathrm{deg}}(t)$, is defined to be the union of all positive dimensional subvarieties $\calY \subseteq \calX$ with $\delta_v(\calY)<t$.
\end{defn}

A subvariety $\calZ\subseteq\calX$ is said to be \emph{vertically optimal} in $\calX$ if, for any subvariety $\calY\subseteq\calX$ containing $\calZ$, one has $\delta_v(\calY)>\delta_v(\calZ)$. Note that the degeneracy locus $\calX^{\mathrm{deg}}(t)$ can be viewed as the union of positive dimensional \emph{vertically optimal} subvarieties with vertical defect $<t$.

As a possibly infinite union of subvarieties, the degeneracy locus is indeed algebraic by \cite[Thm.~1.8]{Gao_Betti}. Gao's proof works for subvarieties of $\gothA_g$. For reader's convenience, we give a proof below following a similar idea. First, we make a convenient definition.

\begin{defn}
    Let $\calY\subseteq\calA$ be a locally closed subvariety. A \emph{type} of $\calY$ is a triple $(\calA_1,\calB_1,N)$ that consists of abelian subschemes $\calB_1\subseteq\calA_1\subseteq\calA|_{\pi(\calY)}$ and $N\in \NN$ such that
    \begin{enumerate}
        \item The special closure of $\calY$ is $\calA_1+\sigma_1$ for a torsion multi-section $\sigma_1\subseteq \calA|_{\pi(\calY)}$.
        \item The weakly special closure $\langle \calY\rangle=\calB_1+\sigma_1'$ for a constant multi-section $\sigma_1'\subseteq \calA|_{\pi(\calY)}$.
        \item The integer $N$ kills the torsion $\sigma_1$, i.e., $N\sigma_1=0$.
    \end{enumerate}
    In this case, we shall say $\calY$ \emph{belongs to the type} $(\calA_1,\calB_1,N)$.
\end{defn}
The idea is to first generalize the finiteness result in Thm.~\ref{thm_GZP} to finiteness of types of vertically optimal subvarieties.

\begin{theorem}\label{thm_finiteness_vertically_optimal}
    Let $\calX\subseteq\calA$ be a locally closed subvariety. There exists a finite set $\Sigma(\calX)$ of types such that every vertically optimal subvariety $\calZ$ of $\calX$ belongs to a type in $\Sigma(\calX)$.
\end{theorem}
\begin{proof}
$\boxed{\text{Case }\calA=\gothA_g}$ Claim: vertically optimal implies weakly optimal. Assuming the claim, the result follows directly from Thm.~\ref{thm_GZP}. Indeed, if $\calZ\subseteq\calX$ is vertically optimal, and let $\calY\subseteq\calX$ contain $\calZ$, by definition and Prop.~\ref{prop_bialgebraic_weaklyspecial}, we have
    \[
    \rdim\calY^{\mathrm{biZar}}-\dim \calY>\rdim\calZ^{\mathrm{biZar}}-\dim \calZ.
    \]
    Since $\pi(\calZ^{\mathrm{biZar}})=\pi(\calZ)^{\mathrm{biZar}}$ and similarly for $\calY'$, we derive from $\pi(\calY)\supseteq\pi(\calZ)$ that 
    \[
    \dim \pi(\calY^{\mathrm{biZar}})\geq\dim \pi(\calZ^{\mathrm{biZar}}).
    \]
    Combining the two inequalities, we see that $\delta_{\mathrm{ws}}(\calY)>\delta_{\mathrm{ws}}(\calZ)$.

 $\boxed{\text{General Case}}$
    Claim: if $\calX=\calU\sqcup\calX_1$ with $\calU$ Zariski open, and the result holds for $\calU,\calX_1$, then we can take $\Sigma(\calX):=\Sigma(\calU)\cup\Sigma(\calX_1)$. Indeed, if $\calZ$ is a vertically optimal subvariety of $\calX$, then $\calZ$ must be Zariski closed. If $\calZ\subseteq\calX_1$, then $\calZ$ is vertically optimal in $\calX_1$ by definition. Otherwise, $\calZ\cap \calU$ is vertically optimal in $\calU$. So the union $\Sigma(\calU)\cup\Sigma(\calX_1)$ takes into account all possible vertically optimal subvarieties of $\calX$.

    Now take a moduli sieve $\iota:\calA\rightarrow \gothA_{g}$. By Chevalley's upper semi-continuity theorem~\cite[Thm.~13.1.3]{EGA_IV3}, there exists a nonempty open subset of $\calU_0\subseteq\iota(\calX)$ such that for any scheme-theoretic point $y\in \calU_0$, the fiber of $\iota|_{\calX}$ over $y$ has pure dimension $\dim \calX-\dim\iota(\calX)$. By the claim, it suffices to work on $\calU:=\iota|_{\calX}^{-1}(\calU_0)$ using Noetherian induction. Without loss of generality, we assume $\iota|_{\calX}$ has pure dimension $r$ over any point of $\iota(\calX)$. In particular, if $\calY'$ is a subvariety of $\iota(\calX)$ and $\calY$ is an irreducible component of $\iota|_{\calX}^{-1}(\calY')$, then using Lem.~\ref{lem_relative_dim_of_ws_closure_under_moduli_sieve},
    \[
    \delta_v(\calY)=\rdim\langle \calY\rangle -\dim\calY =\rdim\langle \calY'\rangle-\dim\calY'-r=\delta_v(\calY')-r.
    \]

    Next, we observe that if $\calZ$ is vertically optimal in $\calX$, then $\calZ$ must be a component of $\iota|_{\calX}^{-1}(\iota(\calZ))$. Indeed, let $\calY$ be a component of $\iota|_{\calX}^{-1}(\iota(\calZ))$ containing $\calZ$. By Lem.~\ref{lem_relative_dim_of_ws_closure_under_moduli_sieve}, 
    \[
    \rdim\langle \calZ \rangle =\rdim\langle \iota(\calZ) \rangle =\langle \calY\rangle.
    \]Since $\delta_v(\calY)\geq \delta_v(\calZ)$, we must have $\dim \calY=\dim\calZ$ and $\calY=\calZ$.

    As a consequence, the vertically optimal subvarieties of $\calX$ corresponds to those of $\iota(\calX)$ in the following way:
     \begin{enumerate}
         \item If $\calZ$ is vertically optimal in $\calX$, then $\iota(\calZ)$ is vertically optimal in $\iota(\calX)$.
         \item If $\calZ'$ is vertically optimal in $\iota(\calX)$, then any irreducible component of $\iota|_{\calX}^{-1}(\calZ')$ is vertically optimal in $\iota(\calX)$.
     \end{enumerate}

Finally we construct $\Sigma(\calX)$ by pulling back the triples in the set $\Sigma(\iota(\calX))$ which is defined in the case $\calA=\gothA_g$. More precisely, let $\Sigma(\calX)$ be the set of $(\calA_1,\calB_1,N)$ with the property that $(\iota(\calA_1),\iota(\calB_1),N)\in \Sigma(\iota(\calX))$ and $\calA_1$ is a component of $\iota^{-1}(\iota(\calA_1))$. Then it is immediate to check that $\Sigma(\calX)$ has the required property. 
\end{proof}

With the above result, we can conclude the proof as in \cite[\S7.2]{Gao_Betti}.

\begin{theorem}\label{thm_degeneracy_Zariski_closed}
 For $t\in \ZZ$, the degeneracy locus $\calX^{\mathrm{deg}}(t)$ for a locally closed subvariety $\calX\subseteq\calA$ is Zariski closed in $\calX$. 

\end{theorem}
\begin{proof}

Take a type $(\calA_1,\calB_1,N)\in \Sigma(\calX)$ as defined in Thm.~\ref{thm_finiteness_vertically_optimal}. Consider the composition
\[
\varphi:[N]^{-1}\calA_1\xrightarrow{[N]} \calA_1\rightarrow \calA_1/\calB_1 \rightarrow \gothA_{g_1-g'_1}
\]where $[N]^{-1}\calA_1$ is the preimage of $\calA_1$ under $[N]:\calA\rightarrow \calA$, $g_1,g_1'$ are the respective relative dimensions of $\calA_1,\calB_1$, and the last morphism is a moduli sieve. Irreducible components of fibers of $\varphi$ are weakly special of relative dimension $g_1'$ by Lem.~\ref{lem_weakly_special_char}. Denote the restriction of $\varphi$ onto $\calX\cap[N]^{-1}\calA_1$ as $\varphi|_{\calX}$. Let 
\[
Z=Z(\calA_1,\calB_1,N):=\{x\in \calX\cap[N]^{-1}\calA_1 \mid \dim_x \varphi|_{\calX}^{-1}(\varphi(x))>\max \{0, g_1'-t\}\}
\] which is Zariski closed by Chevalley's upper semi-continuity theorem \cite[Thm.~13.1.3]{EGA_IV3}. Then $Z\subseteq \calX^{\mathrm{deg}}(t)$ since the vertical defect of any (positive dimensional) irreducible component of fiber $\varphi|_{\calX}^{-1}(\varphi(x))$ is less than
\[
g_1'-\max\{0,g_1'-t\}\leq \min\{g_1',t\}\leq t.
\]
Now the union of such $Z$'s for the finitely many triples $(\calA_1,\calB_1,N)\in \Sigma(\calX)$ is also Zariski closed, and is indeed $\calX^{\mathrm{deg}}(t)$ since all vertically optimal subvarieties have been taken into account.
\end{proof}

Along the path, we obtain the following natural generalization of Gao's criterion \cite[Thm.~8.1]{Gao_Betti} for the case $t\leq 0$ and \cite[Thm.~2.4]{Gao_ICCM} for the case $\calA=\gothA_g$.

\begin{theorem}\label{thm_criterion_for_Xdeg_equal_X}
    Suppose $\calX$ is not contained in a strict group subscheme of $\calA$. The following are equivalent:
    \begin{enumerate}
        \item $\calX=\calX^{\deg}(t)$.
        \item There exists an abelian subscheme $\calB$ of relative dimension $g'$ such that the composition
    \[
    \varphi: \calA\rightarrow \calA/\calB \rightarrow \gothA_{g-g'}
    \]where $\calA\rightarrow\calA/\calB$ is the quotient map and $\calA/\calB\rightarrow \gothA_{g-g'}$ is a moduli sieve, satisfies 
    \[\dim\varphi(\calX)<\dim \calX -\max \{0,g'-t\}.\]
    \end{enumerate}
\end{theorem}
\begin{proof}
 For sufficiency, note that irreducible components of the fibers of $\varphi$ are weakly special of relative dimension $g'$ by Lem.~\ref{lem_weakly_special_char}. By assumption, any irreducible component $\calY$ of a fiber of $\varphi|_{\calX}$ has dimension $\dim \calY>\max\{0,g'-t\}$. Thus, $\calY$ is positive dimensional and
    \[
    \delta_v(\calY)=\dim\langle \calY\rangle -\dim \calY<g'-\max\{0,g'-t\}=\min\{g',t\}\leq t,
    \]which means $\calY\subseteq \calX^{\mathrm{deg}}(t)$.

    To prove necessity, by the proof of Thm.~\ref{thm_degeneracy_Zariski_closed}, we can express $\calX^{\deg}(t)$ as a finite union of Zariski closed sets of the form $Z(\calA_1,\calB_1,N)$. If $\calX^{\deg}(t)=\calX$, then there is one $(\calA_1,\calB,N)\in \Sigma(\calX)$ with $Z(\calA_1,\calB,N)=\calX$. In particular, we have $\calX\subseteq[N]^{-1}\calA_1$, which implies $\calA_1=\calA$ by assumption. Moreover,  by definition of $Z(\calA,\calB,N)$, the fiber dimension of $\varphi|_\calX$ is greater than $\max \{0,g'-t\}$ at any point $x\in \calX$. Hence we have 
    \[\dim \varphi(\calX)<\dim \calX -\max \{0,g'-t\}. \]
    Thus we are done.
\end{proof}

We remark that the first assumption above ensures that $\calX$ is necessarily dominant over $S$, since the union of $\calA|_{\pi(\calX)}$ with the zero section of $\calA$ is a group subscheme of $\calA$. On the other hand, such assumption is minor as in Pink's Conjecture~\ref{conj_pink}. In defining the degeneracy locus, the choice of the ambient abelian scheme is immaterial, and we may replace  $\calA$ with the smallest group subscheme of $\calA|_{\pi(\calX)}$ containing $\calX$. This will yield an abelian subscheme $\calA'$ of $\calA|_{\pi(\calX)}$ if we further replace $\calX$ with $[N]\calX$ for some positive integer $N$. Then Thm.~\ref{thm_criterion_for_Xdeg_equal_X} applies in the setting $[N]\calX\subseteq\calA'$.






\subsection{The criterion}

Let $\widetilde\calX\subseteq \Lie(\calA/S)$ be a connected component of $\exp^{-1}(\calX)$. We have the following diagram
\[
\begin{tikzcd}
     &\widetilde\calX\subseteq \Lie(\calA/S)  \arrow[r,"df"] \arrow[d,"\exp"] & \Lie(\calB/S)\arrow[d,"\exp"]\\
     & \calX\subseteq \calA \arrow[r,dashrightarrow,"f"]  & \calB
\end{tikzcd}
\]Let $\beta:\calB_\Delta\rightarrow \TT^{2g'}$ be a Betti map on $\calB$ associated to the simply connected open neighborhood  $\Delta\subseteq S$. Consider the composition
\[
\psi_\Delta: \Lie(\calA_\Delta/\Delta)
\xrightarrow{df}\Lie(\calB_\Delta/\Delta)\xrightarrow{\exp}\calB_\Delta\xrightarrow{\beta_\Delta}\TT^{2g'}.
\]Note that $\psi_\Delta$ has complex analytic fibers.

\begin{lemma}\label{lemma_volume_vs_immersion}
   The volume form $\omega^{\wedge d}|_\calX$ is nonzero at $x\in \calX^{\mathrm{sm}}\cap \calA_\Delta$ if and only if 
   \[
   \psi_\Delta|_{\widetilde\calX}:\widetilde\calX\longrightarrow \TT^{2g'}
    \]
   is an immersion at $\tilde x$, where $\tilde x\in \widetilde\calX$ is a point above $x$.
\end{lemma}

Recall that a $C^\infty$-map $\varphi:\calM\rightarrow \calN$ between two differentiable manifolds is an immersion at $p$ if the differential $d\varphi_p:T_p\calM\rightarrow T_{\varphi(p)}\calN$ is injective.

 \begin{proof}
 Note that $\ker (d\psi_\Delta)_{\tilde x}$ is exactly the lift of $\ker\omega_x$. Thus
     \[
    \ker(d\psi_\Delta|_{\widetilde\calX})_{\tilde x}=\ker (d\psi_\Delta)_{\tilde x}\cap T_{\tilde x}\widetilde\calX
     \]is isomorphic to $\ker\omega_x\cap T_x\calX$. The equivalence follows by the discussion in \S\ref{subsec_complex_geometry}, since we know $\omega^{\wedge d}_x|_\calX\neq 0$ precisely when $\ker\omega_x\cap T_x\calX=0.$
\end{proof}

Now we can give a necessary condition for the vanishing of the volume form.
\begin{theorem}\label{thm_volume_form_vanishing}
   Let $\calX\subseteq \calA$ be a subvariety of dimension $d$, and $\omega=\omega(f^*\calL_\calB)$ as above. If $\omega^{\wedge d}|_\calX\equiv 0$, then $\calX^{\mathrm{deg}}(g-g')=\calX$.
\end{theorem}
\begin{proof}
First assume $S\subseteq \AAA_g$ and $\calA=\gothA_g|_S$. Let $\calW:=\ker(df)$, and let $\widetilde\calW$ be the pullback of $\calW$ to $\widetilde S$, where $\widetilde S$ is a component of $u_g^{-1}(S)$.  We need the following algebraicity lemma:

\begin{lemma}\label{lemma_algebraicity}
    If $S$ is bi-algebraic, then the flat $C^\infty$ complex vector bundle $\widetilde\calW$ viewed as a subset of $\CC^g\times \gothH_g$ is algebraic. In general, the vector bundle $\widetilde\calW$ is the restriction onto $\widetilde S$ of an algebraic vector bundle of same rank.
\end{lemma}
\begin{proof}
   
    For the first statement, by o-minimal Chow's theorem \cite[Cor.~4.5]{PS_Analytic09}, it suffices to show that $\widetilde\calW$ is definable in $\RR_{\mathrm{an,exp}}$ (or any o-minimal structure).  This follows from the algebraicity of $\widetilde S$ and local triviality of $\widetilde\calW$ by Prop.~\ref{prop_integrable_via_analytic_subgroup} over the  simply connected definable open subsets of $\widetilde S$.  For the second statement, Prop.~\ref{prop_bialgebraic_weaklyspecial} implies that the endomorphism groups of $\calA$ and $\calA^{\mathrm{biZar}}$ coincide. In particular, the $\RR$-homomorphism $\calA\dashrightarrow \calB$ can be extended to an $\RR$-homomorphism of abelian schemes $f_1:\calA_1\dashrightarrow \calB_1$ over $S^{\mathrm{biZar}}$ and we have 
    \[
    \calW=\ker(df)= \ker (df_1)|_S.
    \]The conclusion follows.
\end{proof}

    For $x\in \calX^{\mathrm{sm}}$, the fiber of $\psi_\Delta|_{\widetilde\calX}$ over $\psi_\Delta(\tilde x)$ is positive dimensional by Lem.~\ref{lemma_volume_vs_immersion}, and complex analytic. Let $\widetilde F$ be a connected component of the fiber containing $\tilde x$ and $F:=\exp(\widetilde F)\subseteq \calA$.  We lift $F$ to an analytic irreducible subset of $\CC^g\times \gothH_g$ by lifting $\Delta\subseteq S$ to $\widetilde\Delta\subseteq\widetilde S$, and by slight abuse of notation, still denote it by $\widetilde F$. By  Cor.~\ref{cor_AxSchanuelVertical}, we get
    \[
    \delta_{v}(F^{\mathrm{Zar}})\leq \rdim\widetilde{F}^{Zar} -\dim \widetilde{F}.
    \]But $\widetilde F$ is contained in $\widetilde \calW+\sigma$ for some horizontal section $\sigma=(r_0,s)\in \RR^{2g}\times \widetilde S\cong \CC^g\times \widetilde S$ with $r_0$ fixed. Since $\widetilde\calW$ is algebraic by Lem.~\ref{lemma_algebraicity}, we get $\widetilde{F}^{Zar}\subseteq \widetilde \calW+\sigma$. This implies
    \[
    \rdim\widetilde{F}^{Zar}\leq \rank \calW=g-g'.
    \]Thus, $\delta_v(F^{\mathrm{Zar}})<g-g'$ and $x\in F^{\mathrm{Zar}}\subseteq \calX^{\mathrm{deg}}(g-g')$. Therefore $\calX^{\mathrm{sm}}\subseteq\calX^{\mathrm{deg}}(g-g')$ and by Thm.~\ref{thm_degeneracy_Zariski_closed}, we conclude that $\calX^{\mathrm{deg}}(g-g')=\calX$.

For a general abelian scheme $\calA/S$, take a moduli sieve $\iota:\calA\rightarrow \gothA_g$. If $\iota|_{\calX}$ is not generically finite, then $\calX^{\mathrm{deg}}(0)=\calX$ and hence $\calX^{\mathrm{deg}}(g-g')=\calX$. Assume without loss of generality that $\iota|_{\calX}$ is quasi-finite. Consider the corresponding abelian schemes $\iota(\calA),\iota(\calB)$. It is easy to see that the particular choice of ample $\calL_\calB$ does not matter. So we take a relatively ample $\calL_{\iota(\calB)}$ on $\iota(\calB)$ and use $\calL_\calB:=\iota^*\calL_{\iota(\calB)}$. Then the corresponding Betti form on $\iota(\calA)$ pulls back to the Betti form $\omega$ on $\calA$. The measure $\omega^{\wedge d}|_{\calX}\equiv 0$ implies the same for the corresponding measure on $\iota(\calX)$. So we get $\iota(\calX)^{\mathrm{deg}}(g-g')=\iota(\calX)$. As in the proof of Thm.~\ref{thm_finiteness_vertically_optimal}, we have that for a vertically optimal $\calZ'$ of $\iota(\calX)$, any irreducible component of $\iota|_{\calX}^{-1}(\calZ')$ is vertically optimal of the same vertical defect. From this we conclude that $\calX^{\mathrm{deg}}(g-g')=\calX$.
\end{proof}

%% file: s6.tex
In this section, we give a brief account of adelic line bundles in the sense of Yuan--Zhang \cite{YuanZhang_ALB}, with a focus on abelian schemes, to introduce the notations for the sequel.  We refer to their foundational paper for definitions and more details.

\subsection{Overview}
The Yuan--Zhang theory, which builds upon classical Arakelov theory, provides a unified framework for intersection theory on line bundles over general quasi-projective varieties defined over a field. It begins with the category of metrized line bundles on the Berkovich analytic space associated with a quasi-projective variety over a field equipped with a set of valuations, and then trims this eventually to the subcategory of \emph{integrable adelic} line bundles. Specifically, metrized line bundles that arise directly from geometric models—such as hermitian line bundles on projective arithmetic varieties—are referred to as \emph{model adelic} line bundles. In general, \emph{adelic} line bundles are conceived as limits of model adelic line bundles, with the limits taken in an appropriate boundary topology. The most important subcategory of \emph{integrable adelic} line bundles, consists of those adelic line bundles that are applicable in intersection theory. These are defined as the difference between two nef adelic line bundles, analogous to the construction of Lebesgue integrals, where nef adelic line bundles play a role similar to that of non-negative functions. Furthermore, the association of the category of adelic line bundles to the quasi-projective variety has natural functoriality under morphisms of the underlying varieties, maintaining compatibility with all relevant substructures.

Now let us delve into more details. Let $\bar K:=(K,\Sigma)$ denote a pair consisting of a field $K$ and a set $\Sigma$ of valuations on $K$ in one of the following cases:
\begin{enumerate}
    \item \textbf{geometric}: $K$ is a field and $\Sigma$ consists of one element, the trivial valuation on $K$.
    \item \textbf{archimedean local}: $K$ is the field $\CC$ of complex numbers and $\Sigma$ consists of one element, the usual norm on $\CC$.
    \item \textbf{arithmetic}: $K$ is a number field and $\Sigma$ is the set of all normalized valuations of $K$, including the trivial valuation.
\end{enumerate}
 Let $X$ be a quasi-projective variety over the field $K$ and let $L$ be a usual line bundle on $X$. For a pair $\bar K=(K,\Sigma)$ as above, there is an associated Berkovich analytic space $X^{\mathrm{an}}=X^{\mathrm{an}}_{\Sigma}$ and the natural extension of $L$ to $L^{\mathrm{an}}$ on $X^{\mathrm{an}}$. A metrized line bundle is a pair $(L,\lVert\cdot\rVert)$ with a continuous metric $\lVert\cdot\rVert$ on $L^{\mathrm{an}}$.
We shall not give a precise definition of adelic line bundles here. Roughly speaking, an adelic line bundle extending $L$ is a metrized line bundle which can be regarded as a suitable completion of $L$ at various places in $\Sigma$. If there is no ambiguity, such an extension is denoted by $\overbar{L}$. The group of isomorphism classes of adelic line bundles is denoted by $\hPic(X/\bar K)$, with group law given by tensor products and written additively. The subset consisting of nef (resp. integrable) adelic line bundles is denoted by $\hPic_{\mathrm{nef}}(X/\bar K)$ (resp. $\hPic_{\mathrm{int}}(X/\bar K)$). Given a morphism $f:X\rightarrow Y$ over $K$, the pullback map $f^*: \hPic(Y/\bar K)\rightarrow \hPic(X/\bar K)$ is naturally defined and preserves nefness and integrability. 
 
 There are also natural relations among three cases above. To introduce these, we need to be more careful with the notation for $\bar K$. By slight abuse of notation, we simplify the pair $(K,\Sigma)$ in three cases as $K$ (geometric), $\calO_\CC$ (archimedean local), $\calO_K$ (arithmetic), respectively. 
 There are natural \emph{localization} maps from arithmetic to archimedean local ($K=\text{number field}$ and $\sigma:K\rightarrow \CC$ is an embedding)
\[
\hPic(X/\calO_K)\xrightarrow{\text{loc.}} \hPic(X_\sigma/\calO_{\CC})
\]and from arithmetic/archimedean local to geometric ($K=\text{number field/}\CC$)
\[
\hPic(X/\calO_K)\xrightarrow{\text{loc.}} \hPic(X/K)
\]given by forgetting the metrics over the unrelevant places. There is a \emph{base change} map in the geometric setting ($K=\text{number field}$ and $\sigma:K\rightarrow \CC$ is an embedding)
\[
\hPic(X/K)\rightarrow \hPic(X_\sigma/\CC).
\]In the rest of the article, we shall only use the notation $\bar K$ if a uniform treatment is possible.

Intersection theory is developed in the geometric and arithmetic contexts:
\begin{enumerate}
    \item \textbf{geometric}: In this case, adelic line bundles arise as limits of ordinary line bundles on projective models of $X$ over $K$. The intersection number is obtained via this limiting process. This framework generalizes the traditional intersection theory of line bundles in algebraic geometry to a quasi-projective setting.
    
    \item \textbf{arithmetic}: Here, adelic line bundles are limits of hermitian line bundles on projective models of $X$ over $\calO_K$. The intersection number is similarly derived through a limiting process. This theory extends Arakelov's intersection theory of arithmetic line bundles to a quasi-projective context.
\end{enumerate} 
Denote the dimension of a relevant model of $X$ by $d$, whence in the geometric (resp. arithmetic) case, we have $d=\dim X$ (resp. $d=\dim X+1$). We call $d$ the \emph{model dimension} of $X$. By extending $\RR$-linearly, \emph{intersection theory} gives a symmetric multi-linear pairing 
\begin{equation}
    \begin{split}
        [\cdot]_X:\hPic_{\mathrm{int}}(X/\bar K)^d_\RR &\longrightarrow \RR\\
        (\overbar{L}_1,\cdots,\overbar{L}_d) &\longmapsto [\overbar{L}_1\cdots \overbar{L}_d]_X.
    \end{split}
\end{equation}
The pairing satisfies the following basic properties:
\begin{enumerate}
    \item If  $\overbar{L}_1,\cdots,\overbar{L}_d\in \hPic_{\mathrm{nef}}(X/\bar K)$, then $[\overbar{L}_1\cdots \overbar{L}_d]_X\geq 0$.
    \item When $K$ is a number field and $\sigma:K\rightarrow \CC$ is an embedding, geometric intersection theory on $\hPic_{\mathrm{int}}(X/K)$ agrees with that on $\hPic_{\mathrm{int}}(X_\sigma/\CC)$ through base change. 
    \item  \emph{Projection formula} \cite[Prop.~4.1.2]{YuanZhang_ALB}: Let $f:X\rightarrow Y$ be a dominant morphism of quasi-projective varieties over $K$. Let $\overbar{L}_1,\cdots,\overbar{L}_d\in \hPic_{\mathrm{int}}(Y/\bar K)$. Then 
    \begin{equation}\label{eqn_proj_formula}
        [f^*\overbar{L}_1\cdots f^*\overbar{L}_d]_X=\deg f\cdot [\overbar{L}_1\cdots \overbar{L}_d]_Y.
    \end{equation} Here, if $f$ is generically finite, $\deg f$ is the degree of the extension between the function fields of $X$ and $Y$; otherwise, $\deg f=0$. 
\end{enumerate}
When $f:X\rightarrow Y$ is clear in the context, and $\overbar{L}_1,\cdots, \overbar{L}_d\in \hPic_{\mathrm{int}}(Y/\bar K)$, we shall write $[\overbar{L}_1\cdots \overbar{L}_d]_X$ for $[f^*\overbar{L}_1\cdots f^*\overbar{L}_d]_X$.

\begin{defn}
   Let $K$ be a number field. The \emph{height function} associated to an arithmetic adelic line bundle $\overbar{L}\in \hPic_{\mathrm{int}}(X/\calO_K)_\RR$  is  a function  $h_{\overbar{L}}:X(\Ka)\rightarrow \RR$ defined by
\[
 h_{\overbar{L}}(x):=\frac{[\overbar{L}]_{x'}}{[\kappa(x'):K]}
\]where $x\in X(\Ka)$ is an algebraic point and $x'\in X$ is the scheme-theoretic point of $x$.
\end{defn}

For instance, on a projective space $\PP^N$ over a number field $K$, the naive adelic line bundle $\overbar{\calO(1)}$ is the adelic metric on $\calO(1)$ whose associated height function $h_{\overbar{\calO(1)}}$ is the usual naive height function.


\subsection{Volume and bigness}
We work in the geometric/arithmetic setting for this subsection. Let $\overbar{L}\in\hPic(X/\bar K)$ and let $\lVert\cdot\rVert$ be the associated metric on $L^{\mathrm{an}}$.

\begin{defn}
 The set of \emph{effective sections}, denoted by $\widehat{H}^0(X,\overbar{L})$, is defined to be the subset of global sections $s\in H^0(X, L)$ such that $\lVert s^{\mathrm{an}}(x)\rVert\leq 1$ for any $x\in X^{\mathrm{an}}$, where $s^{\mathrm{an}}$ is the section on $L^{\mathrm{an}}$ corresponding to $s$. Define the quantity
    \begin{equation*}
        \hat h^0(X,\overbar{L}):=\begin{cases}
        \dim_K \widehat H^0(X,\overbar{L}) & \text{in geometric setting}\\
            \log\#\widehat H^0(X,\overbar{L}) & \text{in arithmetic setting}
        \end{cases}
    \end{equation*}
\end{defn}

The existence of nontrivial effective sections has the following implication on height.

\begin{lemma}\label{lem_effective_sec_implies_nonnegative_generically}
    Let $K$ be a number field and let $\overbar{L}\in\hPic_{\mathrm{int}}(X/\calO_K)$. If there is a nonzero effective section $s\in \widehat H^0(X,m\overbar{L})$ for nonzero $m\in\NN$, then $h_{\overbar{L}}(x)\geq 0$ for any $x\in X(\Ka)$ with $s(x)\neq 0$.
\end{lemma}
\begin{proof}
    Let $x\in X(\Ka)$ be a point with $s(x)\neq0$ and let $x'\in X$ be the scheme-theoretic point of $x$. By definition, we have
    \[
    mh_{\overbar{L}}(x)=\frac{[m\overbar{L}]_{x'}}{[\kappa(x'):K]}=\frac{1}{[\kappa(x'):K]}\sum_{v\in \Sigma_K}\sum_{y\in x'_v}-[\kappa(y):K_v]\log\lVert s(y)\rVert
    \]where $\Sigma_K$ denotes all the places of $K$ and $x'_v$ is the finite set of points of $X^{\mathrm{an}}$ given by the image of $x'\times_{\Spec K} \Spec K_v$ in $X_{K_v}$. Then we have $h_{\overbar{L}}(x)\geq0$ since $\lVert s(y) \rVert\leq 1$.
\end{proof} 
 
 \begin{defn}
 The \emph{volume} of $\overbar{L}\in \hPic(X/\bar K)$ is defined as
     \[
\widehat \vol(X,\overbar{L}):=\lim_{m\rightarrow \infty}\frac{d!}{m^d}\widehat h^0(X,m\overbar{L})
\]
where the limit exists by \cite[Thm. 5.2.1]{YuanZhang_ALB}.  More generally, the volume of a $\QQ$-adelic line bundle $\overbar L\in \hPic(X/\bar K)_\QQ$ is defined as $\frac{1}{N^d}\widehat \vol(X,N\overbar{L})$ for any $N\in \NN$ such that $N\overbar{L}\in \hPic(X/\bar K)$.

If $\widehat \vol(X,\overbar{L})>0$, we say $\overbar{L}$ is \emph{big}.
\end{defn}
 


For \emph{integrable} adelic line bundles, volume is closely related to intersection theory in the following way by \cite[Thm. 5.2.2]{YuanZhang_ALB}:
\begin{theorem}\label{thm_hilbert_samuel_bigness}
 For $\overbar{L},\overbar M\in\hPic_{\mathrm{nef}}(X/\bar K)_\QQ$, the following holds:
    \begin{enumerate}
        \item Arithmetic Hilbert--Samuel theorem: $\widehat\vol(X,\overbar{L})=[\overbar{L}^{d}]_X$.
        \item Siu--Yuan's Bigness theorem: $\widehat\vol(X,\overbar{L}-\overbar M)\geq [\overbar{L}^{d}]_X-d\left[\overbar{L}^{d-1}\overbar M\right]_X$.
    \end{enumerate}
\end{theorem}

\begin{cor}
     Let $f:X\rightarrow Y$ be a dominant morphism of quasi-projective varieties over $K$. For $\overbar{L}\in\hPic_{\mathrm{nef}}(Y/\bar K)$, the pullback $f^*\overbar{L}$ is big if and only if $f$ is generically finite and $\overbar{L}$ is big on $Y$.
\end{cor}
\begin{proof}
    This follows directly from arithmetic Hilbert--Samuel and the projection formula \eqref{eqn_proj_formula}. 
\end{proof}

\subsection{Invariant metrics}
Now let $\calA/S$ be an abelian scheme over a quasi-projective normal variety $S$ over $K$. We are interested in certain metric on symmetric line bundles that is \emph{invariant} under the group structure. The starting point is the following existence result

\begin{theorem}\label{thm_invariant_alb}
    For any $\calL\in\Pic_{0}(\calA/S)$, there is a unique integrable adelic line bundle $\overbar{\calL}\in\hPic_{\mathrm{int}}(\calA/\bar K)$ extending $\calL$ such that $[l]^*{\overbar{\calL}}\cong l^2\overbar{\calL}$. If $\calL$ is ample, then $\overbar{\calL}$ is nef as an adelic line bundle. 
\end{theorem}
\begin{proof}
     In the geometric/arithmetic case, the existence and uniqueness are stated in \cite[Thm.~6.1.3]{YuanZhang_ALB} with $\epsilon=1$. Its proof by Tate's limiting argument, and the second statement is given in Thm.~6.1.1 of loc. cit. Localization gives the archimedean local case. 
\end{proof}

By extending linearly, we can associate a unique invariant integrable $\RR$-adelic line bundle to any symmetric $\RR$-line bundle $\calL\in\Pic_{0}(\calA/S)_\RR$. If $\calB/S$ is another abelian scheme and $f\in \Hom(\calA,\calB)_\RR$, define $f^*\overbar{\calL}:=\overbar{f^*\calL}$. The following lemma concerning continuity of intersection numbers with respect to $\RR$-pullbacks will be useful to us:


\begin{lemma}\label{lem_continuity_intersection_pullback}
    Let $\calX\subseteq \calA$ be a subvariety of model dimension $d$. Let $\calL_1,\cdots,\calL_i\in \Pic_{0}(\calB/S)_\RR$ and $\overbar \calM_1,\cdots,\overbar \calM_j\in \hPic_{\mathrm{int}}(\calA/\bar K)_\RR$ with $i+j=d$. In the geometric/arithmetic setting, the function given by intersection numbers
    \[
    \begin{split}
        \Hom(\calA,\calB)_{\RR}&\longrightarrow \RR\\
        f&\longmapsto [(f^*\overbar\calL_1)\cdots (f^*\overbar\calL_i) \overbar\calM_1\cdots\overbar \calM_j]_\calX
    \end{split}
    \]is continuous and homogeneous of degree $2i$.
\end{lemma}
\begin{proof}
The intersection pairings are continuous by definition when restricted to the subspace spanned by $\{f^*\overbar\calL_i,\overbar\calM_j:f,i,j\}$, which is finite dimensional by \eqref{eqn_quadratic_pullback_general}. So the result follows from the continuity of the pullback map as a quadratic form \eqref{eqn_starR}. It is homogeneous of degree $2i$ since $(nf)^*\overbar\calL_i=n^2 f^*\overbar\calL_i$ for any $n\in \RR$.
\end{proof}

For later purpose, we use the following convention to avoid ambiguity:
\begin{enumerate}
    \item \textbf{geometric}: the invariant adelic line bundle is denoted by $\widetilde{\calL}\in \hPic(\calA/K)_\RR$.
    \item \textbf{archimedean local/arithmetic}: the invariant adelic line bundle is denoted by $\overbar{\calL} \in \hPic(\calA/\calO_K)_\RR$.
\end{enumerate}

In the arithmetic case, the height $h_{\overbar{\calL}}$ is the fiber-wise \emph{N\'eron--Tate height} associated to $\calL$, and can be equivalently defined by \emph{Tate's limit argument}. Specifically, if $\calL\in \Pic_0(\calA/S)$ is ample, embedding $\calA$ in some projective space $\PP^N$ such that $[\calO(1)|_{\calA}]=[\calL]^n\in\Pic_0(\calA/S)$ for some $n\in \NN$, then
\[
n\cdot h_{\overbar\calL}(x)=\lim_{l\rightarrow \infty}l^{-2}\cdot h_{\PP^N}([l]x)
\]where $h_{\PP^N}$ is the naive height and $x\in \calA(\Ka)$. Extending linearly gives the definition for general $\calL$. We have the following convenient lemma

\begin{lemma}\label{lem_nef_equals_nonnegative_height}
    For $\calL\in\Pic_{0}(\calA/S)_\RR$, $\calL\geq 0$ if and only if $h_{\overbar\calL}\geq0$. 
\end{lemma}
\begin{proof}

If $\calL$ is  nef, taking an ample $\calL_0\in\Pic_0(\calA/S)$, then $\calL_\epsilon:=\calL+\epsilon \calL_0$ is ample for any $\epsilon>0$, and hence can be written as a positive linear combination of ample integral line bundles by Prop.~\ref{prop_numerical_property} (4). Then $ h_{\overbar\calL_\epsilon}\geq 0$. Letting $\epsilon\rightarrow 0$, we get $h_{\overbar\calL}\geq 0$.

Conversely, if $\calL$ is not nef, then there is a curve $C$ on a fiber such that $\deg(\calL|_C)<0$. Then the height $ h_{\overbar\calL}|_C= h_{\overbar{\calL}|_C}$ on $C$ can be arbitrarily negative since $-\calL|_C$ is ample.
\end{proof}

As a consequence, if $\calL\equiv\calL'$ for $\calL'\in\Pic_0(\calA/S)_\RR$, then $\hat h_\calL=\hat h_{\calL'}$. In other words, the association of N\'eron--Tate height function to $\Pic_0(\calA/S)_\RR$ factors through $\NS(\calA/S)_\RR$.

In the next lemma, we show that a total height function on $\calA$ dominates other height functions in a suitable sense.

\begin{lemma}
     Let $\iota:\calA\hookrightarrow \PP^{N}$ be an embedding and denote the naive height function restricted to $\calA$ by $h$. For $\overbar{\calM}\in \hPic_{\mathrm{int}}(\calA/\calO_K)$, there exists $c>0$ such that $h_{\overbar{\calM}}\leq c\max\{1,h\}$.
\end{lemma}
\begin{proof}
    We give a proof in the case where $\overbar{\calM}$ is model adelic line bundle. General case follows by a limit argument. Let $\calA_0$ be the generic fiber of a projective model of $\calA$ over $\calO_K$ where both $\iota^*\overbar{\calO(1)}$ and $\overbar\calM$ are realized. The divisor on the generic fiber for $\iota^*\overbar{\calO(1)}$ can be taken as $D+E$ where $D$ is ample and $E$ is an effective divisor supported away from $\calA$. Let $D_1$ be a divisor for $\overbar\calM$. Then for sufficiently small $\epsilon>0$, $D-\epsilon D_1$ is ample. In particular, the base locus of $\calO(D-\epsilon D_1+E)$ is in the complement of $\calA$. Thus, 
    $h-\epsilon h_{\overbar{\calM}}$ is bounded below by a constant on $\calA$. This concludes the proof.
\end{proof}

As an immediate corollary, we see that the choice of total height functions is irrelevant for our purpose of bounded height results.
\begin{cor}\label{cor_choice_total_height}
    For $i=1,2$, let $h_{i}$ be the restricted naive height functions on $\calA$ induced by embedding into projective spaces. There exists $c>0$ such that
    \[
    c^{-1}\max\{1,h_1\}\leq \max\{1,h_2\}\leq c\max\{1,h_1\}.
    \]
\end{cor}

Moreover, we can equivalently use a sum of the fiber-wise N\'eron--Tate height with an ample height on the base as the total height, if we use a result of Silverman--Tate~\cite[Thm.~A]{Silverman_Specialization}.

\subsection{Monge--Amp\`ere measure}\label{subsec_measure}
We work over $K=\CC$ in this subsection. Classically, when $X$ is a complex projective variety of dimension $d$ and $\overbar{L}_1,\cdots,\overbar{L}_d$ are smooth hermitian line bundles, the intersection number $[L_1\cdots L_d]_X$ is equal to the integral over $X^{\mathrm{an}}$ of $c_1(\overbar{L}_1)\wedge\cdots\wedge c_1(\overbar{L}_d)$, where $c_1(\bar L)$ denotes the Chern form of $\bar L$, as a special case of the fact that intersection of cycles in homology is Poincar\'e dual to wedge product in cohomology.
In general for quasi-projective $X$, there is a canonical Radon measure associated to any $(\overbar{L}_1,\cdots,\overbar{L}_d)\in \hPic_{\mathrm{int}}(X/\calO_\CC)^d$, called the \emph{Monge--Amp\`ere measure}, on the complex analytic variety $X^{\mathrm{an}}$, denoted by $c_1(\overbar{L}_1)\cdots c_1(\overbar{L}_d)$, which can be defined by a limit process using \cite[Thm.~2.1]{BT_Capacity}.  A basic property is that the integral of $c_1(\overbar{L}_1)\cdots c_1(\overbar{L}_d)$ over any strict Zariski closed subvariety of $X^{\mathrm{an}}$ is $0$. We refer to \cite[\S3.6]{YuanZhang_ALB} for more details and a discussion of the more general \emph{Chambert-Loir measure} over any complete field with non-trivial valuation.

The following allows us to compute the geometric intersection number via integration by the Monge--Amp\`ere measure:
\begin{theorem}\label{thm_Chernsformula}
Let $\overbar{ L}_1,\cdots,\overbar{ L}_d\in \hPic_{\mathrm{int}}(X/\calO_\CC)$ and let $\widetilde L_1,\cdots,\widetilde L_d\in\tPic_{\mathrm{int}}(X/\CC)$ be the localization image. Then
    \[
   [\widetilde L_1\cdots\widetilde L_d]_X= \int_{X^{\mathrm{an}}}c_1(\overbar{L}_1)\cdots c_1(\overbar{L}_d).
    \]
\end{theorem}
\begin{proof}
   For the current general version, see Guo \cite{Guo_IF}. See also \cite[Lem. 5.5.4]{YuanZhang_ALB} for the special case when the involved adelic line bundles are localized from a number field. 
\end{proof}

Now let $\calL_1,\cdots,\calL_d\in\Pic_{0}(\calA/S)$. Let $\calX\subseteq\calA$ be a subvariety of dimension $d$. The Monge--Amp\`ere measure $c_1(\overbar{\calL_d}|_\calX)\cdots c_1(\overbar{\calL_d}|_\calX)$ has a particularly nice description in terms of the Betti forms. 

 \begin{theorem}\label{thm_MAmeasure_Betti}
      For $\calL_1,\cdots,\calL_d\in \Pic_0(\calA/S)$ and a subvariety $\calX\subseteq\calA$ of dimension $d$, 
      \[
    c_1(\overbar{\calL_1}|_\calX)\cdots c_1(\overbar{\calL_d}|_\calX)=\omega(\calL_1)\wedge \cdots\wedge \omega(\calL_d)
      \]as measures on $\calX^{\mathrm{an}}$.
 \end{theorem}
 \begin{proof}
     The result is stated without proof in \cite[\S6.2.2]{YuanZhang_ALB}. We assume that $\calL_i$'s are ample; the general case follows by linearity. We assume also that $S$ is smooth since the measure is zero on the strict Zariski closed  non-smooth part. The idea is to go back to the construction of the invariant metrics by Tate's limiting argument. Start with any hermitian metric $\lVert \cdot\rVert_{i,0}$ on $\calL_i$ and construct inductively $\lVert \cdot \rVert_{i,k}$ on $\calL_i$ for $k\geq 1$ by the following rule:
     \[
    [2]^*(\calL_i,\lVert\cdot\rVert_{i,k-1})\cong (\calL_i,\lVert \cdot \rVert_{i,k})^{\otimes 4}.
     \]Then $(\calL_i,\lVert \cdot \rVert_{i,k})$ tends to $\calL_i$ with the invariant metric $\lVert \cdot \rVert$. Take nonzero rational sections $s_i$ of $\calL_i$. By ampleness, $g_{i,k}:=-\log\lVert s_i \rVert_{i,k}^2$ are sequences of pluri-subharmonic (psh) functions converging locally uniformly to $g_i:=-\log\lVert s_i \rVert^2$. By \cite[Cor.~2.6]{Demailly_MAoperator}, we know the weak limit
     \[
    c_1(\overbar{\calL_1}|_\calX)\cdots c_1(\overbar{\calL_d}|_\calX):=\lim_{k\rightarrow \infty}\frac{\sqrt{-1}}{2\pi}\partial\bar\partial(g_{1,k}) \wedge\cdots\wedge \frac{\sqrt{-1}}{2\pi}\partial\bar\partial(g_{d,k})
     \]as a measure on $\calX^{\mathrm{an}}$ can be given by the form
     \[
        \frac{\sqrt{-1}}{2\pi}\partial\bar\partial(g_{1}) \wedge\cdots\wedge \frac{\sqrt{-1}}{2\pi}\partial\bar\partial(g_{d})
     \]restricted to $\calX^{\mathrm{an}}$. On the other hand, the Chern form $\frac{\sqrt{-1}}{2\pi}\partial\bar\partial(g_{i})$ of $(\calL_i,\lVert\cdot\rVert)$ is the unique invariant form $\omega(\calL_i)$. So we are done.
    \end{proof}

\begin{cor}\label{cor_intersection_to_form}
    For a nef $\calL\in \Pic_0(\calA/S)_\RR$ and a subvariety $\calX\subseteq\calA$ of dimension $d$, the self-intersection number $[\widetilde\calL^d]_\calX>0$ if and only if $\omega(\calL)^{\wedge d}|_{\calX^{\mathrm{an}}}\not \equiv 0$. 
\end{cor}
\begin{proof}
    According to Thm.~\ref{thm_Chernsformula} and Thm.~\ref{thm_MAmeasure_Betti}, we have
    \[
    [\widetilde\calL^d]_\calX=\int_{\calX^{\mathrm{an}}}\omega(\calL)^{\wedge d}|_{\calX^{\mathrm{an}}}.
    \]By nefness of $\calL$, we know $\omega(\calL)^{\wedge d}|_{\calX^{\mathrm{an}}}$ is semipositive. Hence the above is positive if and only if the form is nontrivial as a measure.
\end{proof}

%% file: s7.tex
In this section, $\calA,\calB$ are abelian schemes, of respective relative dimension $g,g'$, over a normal quasi-projective variety $S$ defined over a number field $K$. We fix ample line bundles $\calL_\calA\in\Pic_{0}(\calA/S)$ and $\calL_\calB\in \Pic_{0}(\calB/S)$. Fix $\overbar M_S\in \hPic_{\mathrm{nef}}(S/\calO_K)$ and let $\overbar\calM:=\overbar\calL_\calA+\pi^*\overbar M_S\in \hPic_{\mathrm{nef}}(\calA/\calO_K)$.  In particular, $h_{\overbar\calM}(x)= h_{\overbar\calL_\calA}(x)+h_{\overbar M_S}(\pi(x))$, for any $x\in\calA(\Ka)$. We refer to $ h_{\overbar\calL_\calA}$ as the \emph{relative height} and $h_{\overbar\calM}$ as the \emph{total height}.

\begin{defn}
Let $0<\epsilon<\frac{1}{3}$. The \emph{$\epsilon$-neighborhood} of a point $x\in \calA(\Ka)$ is defined by
\[
C(\epsilon,x):=\left\{x+y\in \calA(\Ka) \mid h_{\overbar\calL_\calA}(y)< \epsilon \max\{1,h_{\overbar{\calM}}(x) \} \right\}.
\]The $\epsilon$-neighborhood of a subset $\Sigma\subseteq\calA$ is defined by  $C(\epsilon,\Sigma):=\bigcup_{x\in \Sigma(\Ka)}C(\epsilon,x).$
\end{defn}

The set of all $\calB$-null $\Ka$-points in $\calA$ is denoted as $\Null(\calB)$. The main result of this section is

\begin{theorem}\label{thm_for_one_quotient}
    Let $\calX\subseteq\calA$ be a subvariety of dimension $d$. Assume that $[(f^*\widetilde\calL_\calB)^d]_\calX>0$ for any $f\in \Hom(\calA,\calB)^\circ_\RR$. Then there exist $\epsilon,c>0$ and an open dense subset $\calU\subseteq\calX$ such that for any $p\in \calU(\Ka)\cap C(\epsilon,\Null(\calB))$, its total height $h_{\overbar\calM}(p)\leq c$.
\end{theorem}

Note that we can and will assume \eqref{eqn_all_endomorphism_appears} without loss of generality by possibly passing to an \'etale cover of $S$ using Lem.~\ref{lem_existlevelstr}.


\subsection{Height lower bound}
In this subsection, we shall establish a lower bound for the total height of a point $p\in \calA(\Ka)$ that is close to be $\calB$-null. We need to approximate the general $\QQ$-homomorphisms in the compact set $\calK(\calA,\calB)$ from Prop.~\ref{prop_compactness} by finitely many $\QQ$-homomorphisms. Denote the subset of $\calK(\calA,\calB)$ consisting of $\QQ$-homomorphisms by
\[
\calK_\NN(\calA,\calB):=\calK(\calA,\calB)\cap\Hom(\calA,\calB)_\QQ.
\]Denote the finite subset consisting of $\QQ$-homomorphisms with denominator $N\in \NN$ by
\[
\calK_N(\calA,\calB):=\calK(\calA,\calB)\cap N^{-1}\Hom(\calA,\calB)=\{f\in\calK(\calA,\calB)\mid Nf\in \Hom(\calA,\calB)\}.
\]

Fix any (equivalent) norm $|\cdot|$ on the finite dimensional vector space $\Hom(\calA,\calB)_\RR$. We observe three basic facts:
\begin{enumerate}
    \item By compactness, there is $\gamma>0$ such that $\gamma^{-1}\leq |f|\leq \gamma$ for any $f\in \calK(\calA,\calB)$.
    \item Let $\delta>0$. By compactness,
    there exists $N_\delta\in \NN$ such that the $\delta$-neighborhood of the set $\calK_{N_\delta}(\calA,\calB)$ covers $\calK(\calA,\calB)$.
    
    \item  There exists $\lambda>0$, such that $f^*\calL_\calB \leq \lambda^2|f|^2\calL_\calA $ for any $f\in \Hom(\calA,\calB)_\RR$. Here, we use the notation $\calL_1\leq\calL_2$ to mean that $\calL_2-\calL_1$ is nef. 
\end{enumerate}
We give a proof of (3). By continuity of \eqref{eqn_starR_for_NS}, there exists an open neighborhood $\Delta$ of $0\in \Hom(\calA,\calB)_\RR$ such that $[\calL_\calA-f^*\calL_\calB]\in \Amp(\calA/S)\subseteq\NS(\calA/S)_\RR$ for any $f\in \Delta$.
Take $\lambda>0$ such that $ \Delta$ contains the ball centered at $0$ of radius $\lambda^{-1}$. Then $(\lambda |f|)^{-1}f\in \bar \Delta$ for any $f\in\Hom(\calA,\calB)_\RR$ and the required property follows from quadraticity.

We shall fix the meaning of $\gamma,\lambda,\delta,N_\delta$ as above for the rest of the section and the exact choice of $\delta$ will be given later.

\begin{prop}\label{prop_lower_bound_for_Bnull}
    Let $x\in \calA(\Ka)$ be a $\calB$-null point. There exists $f\in \calK_{N_\delta}(\calA,\calB)$ such that
    $h_{f^*\overbar\calL_\calB}(x)\leq \lambda^2\delta^2 h_{\overbar\calL_\calA}(x)$.
\end{prop}
\begin{proof}
    Let $g\in \calK_\NN(\calA,\calB)$ with $Ng\in \Hom(\calA,\calB)$ such that $(Ng) (x)=0$. Let $f\in \calK_{N_\delta}(\calA,\calB)$ such that $|f-g|\leq \delta$. Note that $NN_\delta$ is a common denominator for $f,g$. Write $\varphi:=NN_\delta(f-g)$. Then $\varphi(x)=(NN_\delta f)(x)$ and $|\varphi|<NN_\delta \delta$. By the fact (3) above, we have
    \[
    \varphi^*\calL_\calB\leq \lambda^2 |\varphi|^2\calL_\calA \leq \lambda^2 (NN_\delta \delta)^2\calL_\calA.
    \]This implies by Lem.~\ref{lem_nef_equals_nonnegative_height} that $h_{\varphi^*\overbar\calL_\calB}(x)\leq \lambda^2(NN_\delta \delta)^2 h_{\overbar\calL_\calA}(x).$
   Since 
    \[
    (NN_\delta)^2 h_{f^*\overbar\calL_\calB}(x)=h_{\overbar\calL_\calB}\left((NN_\delta f)(x)\right)=h_{\overbar\calL_\calB}(\varphi(x))= h_{\varphi^*\overbar\calL_\calB}(x)\leq \lambda^2(NN_\delta \delta)^2 h_{\overbar\calL_\calA}(x),
    \] we are done by cancelling $(NN_\delta)^2$ on both sides.
\end{proof}

 We can relax the condition on $x$ and get a height bound for points that are merely close to be $\calB$-null. Fix $0<\epsilon<\frac{1}{3}$ to be determined later.

\begin{cor}\label{cor_lower_bound_for_close_Bnull}
    Let $x\in \calA(\Ka)$ be a $\calB$-null point. Let $p=x+y\in C(\epsilon,x)$. There exists $f\in \calK_{N_\delta}(\calA,\calB)$ such that
    $ h_{f^*\overbar\calL_\calB}(p)\leq 12\lambda^2 (\delta^2+\epsilon\gamma^2)\left( 1 +  h_{\overbar \calM}(p)\right)$.
\end{cor}
\begin{proof}
    Take $f$ as in Prop.~\ref{prop_lower_bound_for_Bnull}. We apply the quadraticity of N\'eron--Tate heights multiple times. First of all, we have
    \[
    h_{\overbar\calL_\calA}(x)\leq 2  h_{\overbar\calL_\calA}(p)+ 2  h_{\overbar\calL_\calA}(y)\leq 2 h_{\overbar\calL_\calA}(p)+ 2\epsilon + 2\epsilon h_{\overbar\calM}(x).
    \]Since $\epsilon<\frac{1}{3}$, we deduce that
    \begin{equation}\label{eqn_htx}
    h_{\overbar\calL_\calA}(x)<6  h_{\overbar\calL_\calA}(p)+ 2 + 2 h_{\overbar M_S}(\pi(p))\leq 6(h_{\bar\calM}(p)+1).
    \end{equation}By definition, we also get
    \begin{equation}\label{eqn_hty}
    h_{\overbar\calL_\calA}(y)<\epsilon\left(6  h_{\overbar\calL_\calA}(p)+ 2 + 3h_{\overbar M_S}(\pi(p))\right) <6\epsilon(h_{\bar\calM}(p)+1).
    \end{equation}
 Now, by Prop.~\ref{prop_lower_bound_for_Bnull} and Facts (1) and (3) we have
    \[
    \begin{split}
         h_{f^*\overbar\calL_\calB}(p)&\leq 2 h_{f^*\overbar\calL_\calB}(x)+2 h_{f^*\overbar\calL_\calB}(y)\\
         &\leq 2\lambda^2\delta^2 h_{\overbar\calL_\calA}(x)+2\lambda^2\gamma^2 h_{\overbar\calL_\calA}(y).
    \end{split}
    \]Combining it with the other two inequalities \eqref{eqn_htx} \eqref{eqn_hty}, we are done.
\end{proof}

\subsection{Height upper bound}
Let $\calX$ be a subvariety of $\calA$ of dimension $d$. We are going to establish a height upper bound in this subsection for the total height of a general point on $\calX$. By Lem.~\ref{lem_continuity_intersection_pullback}, we derive two immediate consequences on intersection numbers:
\begin{enumerate}
    \item There exists $c_1,c_2>0$ such that for any $f\in \calK(\calA,\calB)$, we have
    \begin{equation}\label{eqn_upperbound_intersection}
        [(f^*\overbar\calL_\calB)^d\cdot \overbar\calM]_\calX\leq c_1 \text{ and } [(f^*\widetilde\calL_\calB)^{d-1}\cdot\widetilde\calM]_\calX\leq c_2.
    \end{equation}
    \item Assuming that $[(f^*\widetilde\calL_\calB)^d]_\calX>0$ for any $f\in \Hom(\calA,\calB)^\circ_\RR$, there exists $c_3>0$ such that for any $f\in \calK(\calA,\calB)$, we have
    \begin{equation}\label{eqn_lowerbound_intersection}
        [(f^*\widetilde\calL_\calB)^d]_\calX\geq c_3.
    \end{equation}
\end{enumerate}
Using these, we have the following uniform height inequality:

\begin{prop}\label{prop_upper_bound}
   Let $\calX\subseteq\calA$ be a subvariety of dimension $d$.  Assume that $[(f^*\widetilde\calL_\calB)^d]_\calX>0$ for any $f\in \Hom(\calA,\calB)^\circ_\RR$. Then there exist $\alpha,\beta>0$ such that for any $f\in \calK_\NN(\calA,\calB)$, there is a dense open subset $\calU_f\subseteq\calX$ satisfying
   \[
    h_{f^*\overbar\calL_\calB}(p)\geq \alpha h_{\overbar \calM}(p)-\beta
   \]for any $p\in \calU_f(\Ka)$.
\end{prop}
\begin{proof}
    Take any rational numbers $\beta>0$ and 
    \[
    0<\alpha<\frac{c_3\beta}{d(c_1+dc_2\beta)}.
    \]Let $f\in \calK_\NN(\calA,\calB)$. Let $\overbar \calN$ be a hermitian line bundle on $\Spec\calO_K$ with $\widehat\deg(\overbar \calN)=1$. By slight abuse of notation, we view $\overbar \calN$ as an adelic line bundle on $\calA$ via pullback. One basic property of $\overbar\calN$ is that if $\overbar\calL_1,\cdots,\overbar\calL_d\in \hPic_{\mathrm{int}}(\calA/\calO_K)$, then
    \begin{equation}\label{eqn_intersect_arith_to_geom}
    [\overbar\calL_1\cdots \overbar\calL_d\cdot\overbar\calN]_\calX=[\widetilde \calL_1\cdots\widetilde \calL_d]_\calX
    \end{equation}  where $\widetilde\calL_i$ is the localization of $\overbar\calL_i$ in $\hPic_{\mathrm{int}}(\calA/K)$; cf. \cite[Lem.~4.4.4]{YuanZhang_ALB}.
    
    Let $\overbar\calL':=f^*\overbar\calL_\calB+\beta\overbar\calN\in \hPic_{\mathrm{nef}}(\calA/\calO_K)_\QQ$. We claim that 
    \begin{equation}\label{eqn_volume_positive}
        \widehat \vol(\calX, \overbar\calL'-\alpha \overbar\calM)>0.
    \end{equation}
    As a consequence, it implies the existence of a nonzero $\QQ$-effective section $s_f$ of $\overbar\calL'-\alpha \overbar\calM$, whence by Lem.~\ref{lem_effective_sec_implies_nonnegative_generically}, we have 
    \[
    0\leq h_{\overbar\calL'-\alpha \overbar\calM}(p) =  h_{f^*\overbar \calL_\calB}(p)-\alpha h_{\overbar\calM}(p)+h_{\beta\overbar\calN}(p)=h_{f^*\overbar \calL_\calB}(p)-\alpha h_{\overbar\calM}(p) +\beta
    \] 
    for any $p\in \calU_f(\Ka)$ with $\calU_f:=\calX\setminus\divv(s_f)$, as required. 

    To prove the claim \eqref{eqn_volume_positive}, we apply first Siu--Yuan's bigness theorem \ref{thm_hilbert_samuel_bigness}:
    \begin{equation}\label{eqn_application_bigness}
         \widehat \vol(\calX, \overbar\calL'-\alpha \overbar\calM) \geq [\overbar\calL'^{d+1}]_\calX - (d+1)\alpha\left[\overbar\calL'^d\overbar\calM\right]_\calX.
    \end{equation}Note that using \eqref{eqn_intersect_arith_to_geom} with the lower bound \eqref{eqn_lowerbound_intersection}, we have
    \begin{equation*}
    \begin{split}
        [\overbar\calL'^{d+1}]_\calX&=[(f^*\overbar\calL_\calB)^{d+1}]_\calX+(d+1)\beta\left[(f^*\overbar\calL_\calB)^d\cdot\overbar\calN\right]_\calX
        \\&=[(f^*\overbar\calL_\calB)^{d+1}]_\calX+(d+1)\beta\left[(f^*\widetilde\calL_\calB)^d\right]_\calX\\
        &\geq (d+1)\beta c_3.
    \end{split}
    \end{equation*}
    Meanwhile, for a similar reason, we have by \eqref{eqn_upperbound_intersection} that
    \begin{equation*}
        \begin{split}
           [(\overbar\calL')^d\cdot\overbar\calM]_\calX &=[(f^*\overbar\calL_\calB)^d\overbar\calM]_\calX+d\beta\left[(f^*\overbar\calL_\calB)^{d-1}\cdot\overbar\calN\cdot\overbar\calM\right]_\calX\\
           &= [(f^*\overbar\calL_\calB)^d\overbar\calM]_\calX+d\beta\left[(f^*\widetilde\calL_\calB)^{d-1}\cdot\widetilde\calM\right]_\calX\\
           &\leq c_1+ d c_2 \beta.
        \end{split}
    \end{equation*}
   Combine the inequalities into \eqref{eqn_application_bigness} and we are done.
\end{proof}

\subsection{Proof of Thm.~\ref{thm_for_one_quotient}}
Now we are ready to prove the main result of this section. The idea is to let $\delta,\epsilon$ tend to $0$ so that the height bounds in Cor.~\ref{cor_lower_bound_for_close_Bnull} and Prop.~\ref{prop_upper_bound} lead to constraints on the total height of $p$ in the intersection. Specifically, pick sufficiently small $\delta, \epsilon>0$ such that 
\[
\alpha':=12\lambda^2(\delta^2+\epsilon\gamma^2)<\alpha.
\]Since $\calK_{N_{\delta}}(\calA,\calB)$ is a finite set, let 
\[
\calU:=\cap_{f\in \calK_{N_{\delta}}(\calA,\calB)}\calU_f\subseteq \calX.
\]Then for any $p\in \calU(\Ka)\cap C(\epsilon,\Null(\calB))$, we have
\[
\alpha h_{\overbar\calM}(p)-\beta \leq h_{f^*\overbar\calL_\calB}(p)\leq \alpha'h_{\overbar\calM}(p)+\alpha'.
\]Therefore, we get $h_{\overbar\calM}(p)\leq c$ with $c=(\alpha'+\beta)/(\alpha-\alpha').$ Hence, we are done.

%% file: s8.tex
In this section, we assemble various parts to finish the proof of the main theorem. Let $S$ be a normal quasi-projective geometrically connected variety over a number field $K$ and let $\calA/S$ be an abelian scheme of relative dimension $g$. Let $\calX\subseteq\calA$ be a subvariety of dimension $d$. Fix an embedding $K^{\mathrm{alg}}\subseteq \CC$.

\subsection{Descent of degeneracy loci}
The following is a consequence of Chow's theorem. 

\begin{prop}\label{prop_descent}
    Assume $\calA/S$ has a level-$n_0$-structure. For $t\in \ZZ$, the degeneracy locus $\calX_\CC^{\mathrm{deg}}(t)$ is defined over $K$. Namely, there exists a Zariski closed subset $\calX^{\mathrm{deg}}(t)\subseteq\calX$ such that $\calX_\CC^{\mathrm{deg}}(t)=\calX^{\mathrm{deg}}(t)_\CC$.
\end{prop}
\begin{proof}
    By Thm.~\ref{thm_finiteness_vertically_optimal}, there exists a finite set $\Sigma(\calX_\CC)$ of triples $(\calA_1,\calB_1,N)$. By Chow's Theorem \cite[Cor.~3.21]{Conrad_Trace} for the primary extension $\CC(S)/K^{\mathrm{alg}}(S)$, the generic fibers of abelian schemes $\calA_1,\calB_1$ are defined over $K^{\mathrm{alg}}(S)$. Meanwhile by assumption, all abelian subvarieties of the generic fiber of $\calA$ are defined over $K(S)$, whence all abelian subschemes of $\calA$ are defined over $K$; cf. \S\ref{subsec_hom}. Thus in the proof of Thm.~\ref{thm_degeneracy_Zariski_closed}, each set $Z(\calA_1,\calB_1,N)$ is defined over $K^{\mathrm{alg}}$ since the quotient map is defined over $K$ and the moduli sieve is defined over a field containing the reflex field and $K$, whence $K^{\mathrm{alg}}$. So $\calX^{\mathrm{deg}}(t)$ is defined over $K^{\mathrm{alg}}$. Notice that if $\calY\subseteq \calA_{K^{\mathrm{alg}}}$ and $\sigma\in\Gal(K^{\mathrm{alg}}/K)$, then $\sigma(\calY)$ and $\calY$ have the same vertical defect. So $\calX^{\mathrm{deg}}(t)$ is stable under the action of $\Gal(K^{\mathrm{alg}}/K)$. This concludes the proof.
\end{proof}

We make the general definition of degeneracy loci over a number field as follows.

\begin{defn}\label{def_degeneracy_over_number_field}
Take a finite \'etale cover $S'\rightarrow S$ by Lem.~\ref{lem_existlevelstr} such that the base change $\calA'\rightarrow S'$ has a level-$n_0$-structure. Let $\calX\subseteq\calA$ be a subvariety and $t\in\ZZ$.
\begin{enumerate}
    \item The $t^{\mathrm{th}}$ degeneracy locus $\calX^{\mathrm{deg}}(t)$ is defined as the image of $\calX'^{\mathrm{deg}}(t)$, given in Prop.~\ref{prop_descent}, under the natural map $\calA'\rightarrow \calA$.  It is independent of the choice of the cover.
    \item  The subvariety $\calX$ is said to be \emph{nondegenerate} (resp. \emph{$t$-nondegenerate}), if $\calX^{\mathrm{deg}}(0)\neq\calX$ (resp. $\calX^{\mathrm{deg}}(t)\neq\calX$).
\end{enumerate}
\end{defn}

Define $\tau(\calX)$ to be the largest integer $t$ such that $\calX$ is $t$-nondegenerate. If the vertical defect of $\calX$ is $t_0$, then $\tau(\calX)\leq t_0$. When the equality holds, we say $\calX$ is \emph{optimally nondegenerate}.

\subsection{Main theorem and its proof}
We can now reap the fruits of our labor from previous sections. Combining Thm.~\ref{thm_criterion_for_Xdeg_equal_X} and Cor.~\ref{cor_intersection_to_form}, we obtain the positivity of geometric self-intersection number of a class of invariant adelic line bundles asssuming nondegeneracy:
\begin{theorem}\label{thm_geometric_intersection_with_nondegenracy}
Let $t\in\NN$. Given any
\begin{enumerate}
    \item  abelian scheme $\calB/S$ of relative dimension $g'\geq g-t$,
    \item  ample $\calL_\calB\in\Pic_0(\calB/S)$, and
    \item  $f\in \Hom(\calA,\calB)_\RR^{\circ}$, i.e., a surjective $\RR$-homomorphism $f:\calA\dashrightarrow \calB$.
\end{enumerate}
If $\calX\subseteq\calA$ is a $t$-nondegenerate subvariety of dimension $d$, then $[(f^*\widetilde\calL_\calB)^d]_{\calX}>0$. 
\end{theorem}

The main theorem of this paper is deduced below.
\begin{theorem}\label{thm_main}
   Let $\pi:\calA\rightarrow S$ be an abelian scheme on a normal quasi-projective variety $S$ defined over a number field $K$. Let $\calL_\calA\in\Pic_0(\calA/S)$ be ample and let $\overbar\calM:=\overbar\calL_\calA+\pi^*\overbar M_S\in \hPic_{\mathrm{nef}}(\calA/\calO_K)$ with $\overbar M_S\in \hPic_{\mathrm{nef}}(S/\calO_K)$.  For $t\in \NN$ and a subvariety $\calX\subseteq\calA$, there exist $\epsilon,c>0$ such that for any $K^{\mathrm{alg}}$-point
\[
    p\in \left(\calX\setminus \calX^{\mathrm{deg}}(t)\right)\cap C(\epsilon,\calA_{\leq t})
    \]the total height $h_{\overbar\calM}(p)\leq c$.
\end{theorem}
As explained in Cor.~\ref{cor_choice_total_height}, we may replace the height $h_{\overbar{\calM}}$ by another arbitrary total height. 
\begin{proof}
    Assume without loss of generality that $\calX$ is $t$-nondegenerate. By Lem.~\ref{lem_finite_target}, there are finitely many possible abelian schemes $\calB_1,\cdots,\calB_r$ of $\calA$ that admit surjective homomorphisms from $A$ and have relative dimension $\geq g-t$. By Thm.~\ref{thm_geometric_intersection_with_nondegenracy}, the conditions in Thm.~\ref{thm_for_one_quotient} are satisfied for any $\calB\in\{\calB_1,\cdots,\calB_r\}$. Thus for each $1\leq i\leq r$, we find $\epsilon_i,c_i>0$ and open dense subsets $\calU_i\subseteq\calX$ such that the total height of $K^{\mathrm{alg}}$-points in $\calU_i\cap C(\epsilon_i,\Null(\calB_i))$ is bounded by $c_i$. 

    If $\calH$ is a flat group subscheme of $\calA$ of relative dimension $\leq t$, there exists $N\in \NN$ such that $N\calH$ is an abelian subscheme. The quotient $\calA/N\calH$ is an abelian scheme of relative dimension $\geq g-t$, which is hence isogenous to some $\calB$ from $\{\calB_1,\cdots,\calB_r\}$. Then $\calH\subseteq\Null(\calB)$. Thus
    $\calA_{\leq t}=\bigcup_{1\leq i\leq r}\Null(\calB_i).$ 
    
    Let $\calU:=\cap_{i}\calU_i$, $\epsilon:=\min\{\epsilon_i\}_i$ and $c:=\max\{c_i\}_i.$ Then we get that the total height of $K^{\mathrm{alg}}$-points in $\calU\cap C(\epsilon,\calA_{\leq t})$ is bounded by $c$. 
    
    For the rest, we show that we can enlarge $\calU$ to contain $\calX\setminus\calX^{\mathrm{deg}}(t)$. Indeed, if $\calX'$ is a closed subvariety of $\calX$, then $\calX'^{\mathrm{deg}}(t)\subseteq\calX^{\mathrm{deg}}(t)$ by definition. In particular, if $\calX'$ is not contained in $\calX^{\mathrm{deg}}(t)$, then $\calX'\neq\calX'^{\mathrm{deg}}(t)$ and $\calX'$ is $t$-nondegenerate. Now suppose $\calX':=\calX\setminus\calU$ is not contained in $\calX^{\mathrm{deg}}(t)$, then replacing $\calX$ by $\calX'$, we can find $\calU',\epsilon',c'$ as above with the required property. We can enlarge $\calU$ to $\calU'\cup\calU$ after replacing $\epsilon$ by $\min\{\epsilon,\epsilon'\}$ and $c$ by $\max\{c,c'\}$. By Noetherian induction, we can thus assume $\calX\setminus\calX^{\mathrm{deg}}(t)\subseteq\calU$. The proof is now complete.
\end{proof}

%% file: s9.tex
In the first part, we apply the main theorem and criterion to study specialization of Mordell--Weil groups. 
After some simple observation, we restrict attention to a particular case, and present a specialization theorem generalizing Silverman's theorem \cite[Thm.~C]{Silverman_Specialization}. 

In the second part, we establish a bounded height analogue of Conj.~\ref{conj_zhang} using basic linear algebra on quadratic forms.

In the end, we propose a conjecture aligned with the philosophy of Pink's Conj.~\ref{conj_pink}, as an optimal hope to generalize the main theorem.


\subsection{Specialization theorem}
Let $S$ be a normal quasi-projective variety over a number field $K$ with function field $F$, and let $\calA/S$ be an abelian scheme with generic fiber $A$. By properness, the group of sections $\calA(S)$ is identified with the group $A(F)$, which is finitely generated by the Lang--N\'eron theorem~\cite{LangNeron_rationalpoints}. For any $s\in S$, we have a specialization map
\[
\spec_s:A(F)\rightarrow \calA_s
\]which associates to any section $\sigma\in A(F)$ the point $\sigma(s)\in \calA_s$.  If $S'\rightarrow S$ is an \'etale cover and $s'$ is a point above $s$, then $\calA_{s'}=\calA_s$ and $\spec_s$ factors through $\spec_{s'}$. 

Note that $\spec_s$ is injective on the subgroup of constant sections. One can ask in general
\begin{question}
    Given a subgroup $\Lambda\subseteq A(F)$, what can be said about the set of closed points $s\in S$ where $\spec_s:\Lambda\rightarrow \calA_s$ fails to be injective?
\end{question}

For the abelian variety $A/F$ and a subgroup $\Lambda\subseteq A(F)$, define
  \[
      \Sigma(A,\Lambda):=\{s\in S(K^{\mathrm{alg}})\mid \spec_s:\Lambda \rightarrow \calA_s \text{ is not injective}.\}
\]We start with the following basic observation.

\begin{lemma}\label{lem_basic_specialization_facts}
Let $A,B,A_1,A_2$ be the generic fibers of respective abelian schemes $\calA,\calB,\calA_1,\calA_2$ over $S$. Suppose $\Lambda\subseteq A(F),\Lambda'\subseteq B(F)$ are subgroups. The following holds true:
\begin{enumerate}
    \item $\Sigma(A,\bar\Lambda)=\Sigma(A,\Lambda)$, where
    \[\bar\Lambda:=\{\sigma\in A(F) \mid \exists N\in \NN\setminus\{0\}, N\sigma\in \Lambda\}.\]
    \item If $\varphi:A\rightarrow B$ is an isogeny, then $\Sigma(A,\Lambda)=\Sigma(B,\varphi(\Lambda))$.
    \item $\Sigma(A\times B,\Lambda\times \Lambda')=\Sigma(A,\Lambda)\cup \Sigma(B,\Lambda')$.
    \item If $(p_1,p_2):A \rightarrow A_1\times A_2$ is an isogeny, then 
    \[
    \Sigma(A,\Lambda)\subseteq\Sigma(A_1,p_1(\Lambda))\cup \Sigma(A_2,p_2(\Lambda)).
    \]
\end{enumerate}
    
\end{lemma}
\begin{proof}

    (1) Suppose $\spec_s:\bar\Lambda\rightarrow \calA_s$ is not injective for a closed point $s\in S$, and $\sigma$ is a nonzero section in the kernel. Then $\sigma$ is non-torsion. Let $N\in \NN\setminus\{0\}$ with $N\sigma\in \Lambda$. Then $N\sigma$ is nonzero and in the kernel of $\spec_s:\Lambda\rightarrow \calA_s$.
    
    (2) We clearly have $\Sigma(A,\Lambda)\subseteq\Sigma(B,\varphi(\Lambda))$. For the other direction, take an isogeny $\psi:B\rightarrow A$ such that $\psi\circ\varphi=[N]$ and use (1).
    
    (3) By definition.
   
    (4) It suffices to assume $A=A_1\times A_2$ by (2). Note that $\Lambda \subseteq p_1(\Lambda)\times p_2(\Lambda)$. The result hence follows from (3).
\end{proof}

Due to (4), it is essential to study the case where $A$ is simple.
In the following, we will consider a particular case.


\begin{defn}
    An abelian variety $A/F$ is said to have \emph{maximal variation} if the induced period map $S\rightarrow \AAA_g$ is generically finite.
\end{defn}


\begin{theorem}\label{thm_spec_maximal_variation}
    If all simple abelian subvarieties of $A_{F^{\mathrm{alg}}}$ have maximal variation and dimension at least $\dim S$, then $\Sigma(A,A(F))$ is contained in the union of a strict Zariski closed subset and a set of bounded height in $S(K^{\mathrm{alg}})$.
\end{theorem}
The maximal variation assumption on the simple components is equivalent to $A$ having no constant part if $\dim S=1$; so the above result implies Silverman's theorem in the number field case. The dimension assumption ensures the intersection is at most just likely. 

\begin{proof}
Without loss of generality, we assume \eqref{eqn_all_endomorphism_appears} by possibly passing to a finite \'etale cover using Lem.~\ref{lem_existlevelstr}. 
Let $\sigma_1,\cdots,\sigma_r\in A(F)$ be a basis of $A(F)_\QQ$ and let $\underline{\sigma}$ be the image of $(\sigma_1,\cdots,\sigma_r):S\rightarrow \calA_S^{r}$. Let $\calB$ be the smallest (necessarily flat) group subscheme containing $\underline{\sigma}$.  Replacing $\sigma_1,\cdots,\sigma_r$ by $N\sigma_1,\cdots,N\sigma_r$ if necessary, we can and do assume $\calB$ is an abelian subscheme of $\calA_S^r$. Write
\[
g:=\rdim\calB, \, d:=\dim S, \, t:=g-d.
\]

Claim: $\underline{\sigma}$ is $t$-nondegenerate (or optimally nondegenerate). Indeed, take any nonzero abelian subscheme $\calB_1$ of $\calB$ and any surjective homomorphism $\varphi:\calB\rightarrow\calB_1$. Let $\iota_a:\calB_1\rightarrow \gothA_{g'}$ be a moduli sieve with base map $\iota:S\rightarrow \AAA_{g'}$. Since $\calB_1$ has maximal variation, $\iota$ is generically finite. So
\begin{equation}\label{eqn_}
    \dim \iota_a(\varphi(\underline{\sigma}))=\dim\underline{\sigma}\geq \dim \underline{\sigma}-\max\{0,g-g'-t\}.
\end{equation}
For the zero map $\calB\rightarrow 0$, we also have
\[
0\geq \dim \underline{\sigma} -\max\{0,g'-t\}=0.
\]
By Thm.~\ref{thm_criterion_for_Xdeg_equal_X}, these imply that $\underline{\sigma}$ is $t$-nondegenerate. 

Now we consider the following types of flat subgroup schemes of $\calA^r_S$:
\[
\calH_{\underline{\lambda}}:=\{(x_1,\cdots,x_r)\in \calA^r_S\mid \sum_i\lambda_ix_i=0\}
\]where $\underline{\lambda}=(\lambda_1,\cdots,\lambda_r)$ is a nonzero sequence of integers. Note that $\calX$ is not contained in $\calH_{\lambda}$ by choice of $\sigma_i$. So $\calH_\lambda\cap\calB$ is a strict flat group subscheme of $\calB$. The assumption on dimension of the simple factors ensures that the relative dimension of $\calH_\lambda\cap\calB$ is at most $g-d$. Hence Thm.~\ref{thm_main} tells us that the $K^{\mathrm{alg}}$-points of
\[
\Sigma_1:=\left(\calX\setminus\calX^{\mathrm{deg}}(t)\right)\cap \cup_{\lambda}\calH_\lambda
\]form a set of bounded total height.

If the specialization map $\spec_s:A(F)\rightarrow \calA_s(K^{\mathrm{alg}})$ is not injective, then there exists a nonzero $\underline{\lambda}\in\NN^r$ such that $\sum_i\lambda_i\sigma_i$ specializes to $0\in \calA_s(K^{\mathrm{alg}})$ for $s\in S(K^{\mathrm{alg}})$, which means
\[
(\sigma_1,\cdots,\sigma_r)_s\in \calX\cap\calH_\lambda.
\]Let $\pi:\calA^r_S\rightarrow S$ be the structure map. Then 
\[
\Sigma(A,A(F))\subseteq \pi\left( \calX^{\mathrm{deg}}(t) \right)\cup  \pi(\Sigma_1).
\]Here $\pi(\Sigma_1)$ is a set of bounded height. The proof is now complete.
\end{proof}

\subsection{Towards Zhang's conjecture}
Let us first recall the setup of Shou-Wu~Zhang's conjecture. Let $\calA\rightarrow C$ be an abelian scheme on a curve $C$ defined over a number field $K$. Let $\langle\cdot ,\cdot \rangle$ denote the  N\'eron--Tate height pairing associated to an ample line bundle $\calL\in \Pic_0(\calA/C)$. Consider a finitely generated torsion free subgroup $\Lambda\subseteq \calA(C)$ with linearly independent generators $\sigma_1,\cdots, \sigma_r$. In his 1998 ICM talk, S.~Zhang~\cite{Zhang_ICM98} defined the function 
 \[
h_{\Lambda}(s):=\det \left( \langle \sigma_i(s),\sigma_j(s)\rangle_{i,j}\right)
 \]for any $s\in C(K^{\mathrm{alg}})$, and proposed Conj.~\ref{conj_zhang}. This conjecture, in the case of constant abelian scheme, is particularly noteworthy as it implies both the Mordell conjecture and the Bogomolov conjecture for curves. We discuss the non-constant case below. It is also worth mentioning that the so-called \emph{relative Bogomolov conjecture}~\cite{DimitrovGaoHabegger_RBC} originates from Conj.~\ref{conj_zhang} in the special case where $\rank\Lambda=1$.

As noted in the ICM address, Zhang's conjecture was initially stated without a dimension assumption. However, Poonen soon found a counterexample involving a section of an elliptic surface. In this example, the section could possess infinitely many torsion points, yielding infinitely many points with $h_\Lambda(s)=0$. The picture here nowadays is clear: the dimensions of the section and the torsion sections add up to the total ambient dimension which made the intersection more than \emph{unlikely}. From this perspective, the bounded height result that follows is seen as optimal.

\begin{theorem}\label{thm_towardszhang}
    Let $\pi:\calA\rightarrow C$ be an abelian scheme on a curve $C$ defined over a number field $K$, and assume $\calA/C$ has no constant part. Let $\Lambda\subseteq \calA(C)$ be a finitely generated torsion free subgroup. There is $\epsilon>0$ such that 
    \[
    \{s\in C(K^{\mathrm{alg}})\mid h_{\Lambda}(s)<\epsilon\}
    \]is a set of bounded height.
\end{theorem}
In particular, by the Northcott property, there are only finitely many closed points $s\in C$ with $[\kappa(s):K]\leq d$ and $h_\Lambda(s)<\epsilon$.
\begin{proof}
 Let $\sigma_1,\cdots,\sigma_r$ be linearly independent generators of $\Lambda$, and let $\underline{\sigma}$ be the image of $(\sigma_1,\cdots,\sigma_r):S\rightarrow \calA_C^{r}$. Let $\calB$ be the smallest (necessarily flat) group subscheme containing $\calX$.  Replacing $\sigma_1,\cdots,\sigma_r$ by $N\sigma_1,\cdots,N\sigma_r$ if necessary, we can and do assume $\calB$ is an abelian subscheme of $\calA_S^r$, with relative dimension denoted by $g$. We also fix an isogeny decomposition $\varphi:\calA\rightarrow \calB\times \calB'$ such that $p_1\circ\varphi|_{\calB}$ is the identity map where $p_1:\calB\times \calB'\rightarrow \calB$ is the projection. For convenience, we take the fiber-wise N\'eron--Tate height on $\calA$ to be induced from two fiber-wise N\'eron--Tate heights on $\calB,\calB'$ through pullback by $\varphi$. By abuse of notation, all N\'eron--Tate heights are denoted by $\hat h$.
 
 As in the proof of Thm.~\ref{thm_spec_maximal_variation}, the curve $\underline{\sigma}$ is optimally nondegenerate. Therefore, by Thm.~\ref{thm_main}, there exists $\epsilon_0>0$ such that the closed points of 
 \[
    \Sigma:=\pi\left(\underline{\sigma}\cap C(\epsilon_0,\calB_{\leq g-1})\right)
 \]have bounded height.

Set $\epsilon:=(\epsilon_0/2\sqrt{r})^r$. \emph{Claim}: if $s\in C(K^{\mathrm{alg}})$ with $\hat h_\Lambda(s)<\epsilon$, then $s\in\Sigma$. Let $s$ be such a point. To simplify notation, denote the matrix $\left(\langle \sigma_i(s),\sigma_j(s)\rangle\right)_{i,j}$ by $M_s$. Consider the quadratic form $q(\underline{x}):=\underline{x}M_s\underline{x}^t$ for $\underline{x}\in \RR^r$. For any $\underline{\lambda}=(\lambda_1,\cdots,\lambda_r)\in\ZZ^r$, notice that
    \[
    q(\underline{\lambda})=\langle  \underline{\lambda}\cdot\underline{\sigma}(s),\underline{\lambda}\cdot\underline{\sigma}(s)\rangle=\hat h(\sum_i \lambda_i \sigma_i(s))\geq 0.
    \] Hence $q(\underline{x})$ is positive semi-definite by continuity.

By linear algebra, we know that if $\lVert \underline{x}\rVert=1$ where $\lVert\cdot\rVert$ is the standard norm, then $q(\underline{x})$ has minimum equal to the smallest eigenvalue (or equally singular value) $t_0$ of $M_s$, achieved by a corresponding eigenvector $\underline{x_0}\in\RR^r$. Since $t_0^r\leq h_\Lambda(s)<\epsilon$, we derive $t_0<\epsilon_0/2\sqrt{r}$. By approximating $\underline{x_0}$ using vectors in $\QQ^r$, we can at least find $\underline{\mu}\in \ZZ^r\setminus\{0\}$ such that
\[
q(\underline{\mu})< \frac{\epsilon_0}{\sqrt{r}} \lVert\underline{\mu}\rVert^2\leq \epsilon_0\max_{i}\{\mu_i\}^2.
\]Without loss of generality, assume $\max_{i}\{\mu_i\}=\mu_1$. This implies that the point
\[
P_1=(\sum_i \mu_i \sigma_i(s)/\mu_1,0,\cdots,0)\in \calA^r_s
\]has height $\hat h(P_1)<\epsilon_0$. Define $Q_1:=p_1\circ\varphi(P_1)\in \calB$. Then $\hat h(Q_1)<\epsilon_0$.

Consider the flat subgroup scheme 
\[
\calH_{\underline{\mu}}:=\{(a_1,\cdots,a_r)\in \calA_S^r\mid \sum_i\mu_i a_i=0\}.
\]Since $\sigma_1,\cdots,\sigma_r$ are linearly independent, $\underline{\sigma}$ is not contained in $\calH_{\underline{\mu}}$. Hence $\calH_{\underline{\mu}}\cap\calB$ is a strict group subscheme in $\calB$. In particular, $\calH_{\underline{\mu}}\subseteq\calB_{\leq g-1}$. Notice that $P_2:=\underline{\sigma}(s)-P_1\in \calH_{\underline{\mu}}$. Consider the projection $Q_2:=p_1\circ\varphi(P_2)=\underline{\sigma}(s)-Q_1\in \calH_{\underline{\mu}}\cap\calB$. Then we see that 
\[
\underline{\sigma}(s)=Q_2+Q_1\in C(\calH_{\epsilon_0,\underline{\mu}}\cap\calB)\subseteq C(\epsilon_0,\calB_{\leq g-1}).
\]In particular, $s\in \Sigma$. The proof is now complete.
\end{proof}

\subsection{A conjecture}
The general Conjecture \ref{conj_pink} of Pink  suggests to consider not only \emph{flat} group subschemes of $\calA$. Denote by
\[
\calA_{(\leq t)}:=\bigcup_{s\in S}\bigcup_{\substack{H\subseteq \calA_s \\ \dim H\leq r}}H
\]with the union taken over all points $s\in S$ and all group subschemes $H\subseteq\calA_s$ of dimension at most $r$. Note that $\calA_{\leq t}\subseteq\calA_{(\leq t)}$ and the inclusion is strict in general. However, there are still some interesting cases with $\calA_{\leq t}=\calA_{(\leq t)}$; for instance, this is the case if $\calA$ is a fiber product of elliptic fibrations over $S$.  We propose the following conjecture which is an idealistic extension of Thm.~\ref{thm_main}.
\begin{conj}\label{conj_height}
     For $t\in \NN$ and a subvariety $\calX\subseteq\calA$, the closed points of
    \[
    \left(\calX\setminus \calX^{\mathrm{deg}}(t)\right)\cap \calA_{(\leq t)}
    \]form a set of bounded total height.
\end{conj}
An analogue containing an $\epsilon$-neighborhood can be stated as well. 
One heuristic reason for the conjecture is that $\calX^{\mathrm{deg}}(t)$ already takes into account all the group subscheme structure in $\calA$. For instance, in Thm.~\ref{thm_main}, it is totally valid to consider a subvariety $\calX$ in a fiber of $\calA/S$ and the theorem in this form provides a weaker result, since only the more coarse subgroup structure of $\calA_\eta$ is considered. We expect a certain generalization of the compactness result in Prop.~\ref{prop_compactness} to be a key step towards the conjecture.

%% file: main.bbl
\begin{thebibliography}{10}

\bibitem{ACM_Image}
J.~Achter and S.~Casalaina-Martin.
\newblock Images of abelian schemes.
\newblock {\em Available on arXiv}, 2023.

\bibitem{BC_PowersofEC}
F.~Barroero and L.~Capuano.
\newblock Linear relations in families of powers of elliptic curves.
\newblock {\em Algebra Number Theory}, 10(1):195--214, 2016.

\bibitem{BC_productsofEC}
F.~Barroero and L.~Capuano.
\newblock Unlikely intersections in products of families of elliptic curves and the multiplicative group.
\newblock {\em Q. J. Math.}, 68(4):1117--1138, 2017.

\bibitem{BC_CurvesinAbelianScheme}
F.~Barroero and L.~Capuano.
\newblock Unlikely intersections in families of abelian varieties and the polynomial {P}ell equation.
\newblock {\em Proc. Lond. Math. Soc. (3)}, 120(2):192--219, 2020.

\bibitem{BT_Capacity}
E.~Bedford and B.~A. Taylor.
\newblock A new capacity for plurisubharmonic functions.
\newblock {\em Acta Math.}, 149(1-2):1--40, 1982.

\bibitem{BMZ_99}
E.~Bombieri, D.~Masser, and U.~Zannier.
\newblock Intersecting a curve with algebraic subgroups of multiplicative groups.
\newblock {\em Internat. Math. Res. Notices}, (20):1119--1140, 1999.

\bibitem{BMZ_Anomalous}
E.~Bombieri, D.~Masser, and U.~Zannier.
\newblock Anomalous subvarieties---structure theorems and applications.
\newblock {\em Int. Math. Res. Not. IMRN}, (19):Art. ID rnm057, 33, 2007.

\bibitem{BMZ_plane08}
E.~Bombieri, D.~Masser, and U.~Zannier.
\newblock Intersecting a plane with algebraic subgroups of multiplicative groups.
\newblock {\em Ann. Sc. Norm. Super. Pisa Cl. Sci. (5)}, 7(1):51--80, 2008.

\bibitem{CGHX_GBC}
S.~Cantat, Z.~Gao, P.~Habegger, and J.~Xie.
\newblock The geometric {B}ogomolov conjecture.
\newblock {\em Duke Math. J.}, 170(2):247--277, 2021.

\bibitem{Conrad_Trace}
B.~Conrad.
\newblock Chow's {$K/k$}-image and {$K/k$}-trace, and the {L}ang-{N}\'{e}ron theorem.
\newblock {\em Enseign. Math. (2)}, 52(1-2):37--108, 2006.

\bibitem{Demailly_MAoperator}
J.-P. Demailly.
\newblock Monge-{A}mp\`ere operators, {L}elong numbers and intersection theory.
\newblock In {\em Complex analysis and geometry}, Univ. Ser. Math., pages 115--193. Plenum, New York, 1993.

\bibitem{Demailly_Complex}
J.-P. Demailly.
\newblock Complex analytic and differential geometry.
\newblock {A}vailable at \url{https://www-fourier.ujf-grenoble.fr/~demailly/manuscripts/agbook.pdf}, 2012.

\bibitem{DimitrovGaoHabegger_UML}
V.~Dimitrov, Z.~Gao, and P.~Habegger.
\newblock Uniformity in {M}ordell-{L}ang for curves.
\newblock {\em Ann. of Math. (2)}, 194(1):237--298, 2021.

\bibitem{DimitrovGaoHabegger_RBC}
V.~Dimitrov, Z.~Gao, and P.~Habegger.
\newblock A consequence of the relative {B}ogomolov conjecture.
\newblock {\em J. Number Theory}, 230:146--160, 2022.

\bibitem{Faltings_DAonAV}
G.~Faltings.
\newblock Diophantine approximation on abelian varieties.
\newblock {\em Ann. of Math. (2)}, 133(3):549--576, 1991.

\bibitem{FC_DAV}
G.~Faltings and C.-L. Chai.
\newblock {\em Degeneration of abelian varieties}, volume~22 of {\em Ergebnisse der Mathematik und ihrer Grenzgebiete (3) [Results in Mathematics and Related Areas (3)]}.
\newblock Springer-Verlag, Berlin, 1990.
\newblock With an appendix by David Mumford.

\bibitem{Gao_APZ}
Z.~Gao.
\newblock A special point problem of {A}ndr\'{e}-{P}ink-{Z}annier in the universal family of {A}belian varieties.
\newblock {\em Ann. Sc. Norm. Super. Pisa Cl. Sci. (5)}, 17(1):231--266, 2017.

\bibitem{Gao_towards}
Z.~Gao.
\newblock Towards the {A}ndre-{O}ort conjecture for mixed {S}himura varieties: the {A}x-{L}indemann theorem and lower bounds for {G}alois orbits of special points.
\newblock {\em J. Reine Angew. Math.}, 732:85--146, 2017.

\bibitem{Gao_Betti}
Z.~Gao.
\newblock Generic rank of {B}etti map and unlikely intersections.
\newblock {\em Compos. Math.}, 156(12):2469--2509, 2020.

\bibitem{Gao_mixed}
Z.~Gao.
\newblock Mixed {A}x-{S}chanuel for the universal abelian varieties and some applications.
\newblock {\em Compos. Math.}, 156(11):2263--2297, 2020.

\bibitem{Gao_HDR}
Z.~Gao.
\newblock Distribution of points on varieties: various aspects and interactions.
\newblock {A}vailable at \url{https://ziyangjeremygao.github.io/articles/hdr.pdf}, 2021.

\bibitem{Gao_ICCM}
Z.~Gao.
\newblock Bigness of the tautological line bundle and degeneracy loci in families of abelian varieties.
\newblock {A}vailable at \url{https://ziyangjeremygao.github.io/articles/ALBBigness.pdf}, 2024.

\bibitem{GH_RMM}
Z.~Gao and P.~Habegger.
\newblock The {R}elative {M}anin-{M}umford {C}onjecture.
\newblock {\em Available on arXiv}, 2023.

\bibitem{GH_principles}
P.~Griffiths and J.~Harris.
\newblock {\em Principles of algebraic geometry}.
\newblock Wiley Classics Library. John Wiley \& Sons, Inc., New York, 1994.
\newblock Reprint of the 1978 original.

\bibitem{EGA_IV3}
A.~Grothendieck.
\newblock \'{E}l\'{e}ments de g\'{e}om\'{e}trie alg\'{e}brique. {IV}. \'{E}tude locale des sch\'{e}mas et des morphismes de sch\'{e}mas. {III}.
\newblock {\em Inst. Hautes \'{E}tudes Sci. Publ. Math.}, (28):255, 1966.

\bibitem{Guo_IF}
R.~Guo.
\newblock An integration formula of {C}hern forms on quasi-projective varieties.
\newblock {\em Available on arXiv}, 2023.

\bibitem{Habegger_BHConAV}
P.~Habegger.
\newblock Intersecting subvarieties of abelian varieties with algebraic subgroups of complementary dimension.
\newblock {\em Invent. Math.}, 176(2):405--447, 2009.

\bibitem{Habegger_BHCtori}
P.~Habegger.
\newblock On the bounded height conjecture.
\newblock {\em Int. Math. Res. Not. IMRN}, (5):860--886, 2009.

\bibitem{HP_ominimality}
P.~Habegger and J.~Pila.
\newblock O-minimality and certain atypical intersections.
\newblock {\em Ann. Sci. \'{E}c. Norm. Sup\'{e}r. (4)}, 49(4):813--858, 2016.

\bibitem{Klingler_Hodge}
B.~Klingler.
\newblock Hodge loci and atypical intersections: conjectures.
\newblock {\em Available on arXiv}, 2017.

\bibitem{KUY_survey}
B.~Klingler, E.~Ullmo, and A.~Yafaev.
\newblock Bi-algebraic geometry and the {A}ndr\'{e}-{O}ort conjecture.
\newblock In {\em Algebraic geometry: {S}alt {L}ake {C}ity 2015}, volume~97 of {\em Proc. Sympos. Pure Math.}, pages 319--359. Amer. Math. Soc., Providence, RI, 2018.

\bibitem{Kuhne_BHCsemiAV}
L.~K\"{u}hne.
\newblock The bounded height conjecture for semiabelian varieties.
\newblock {\em Compos. Math.}, 156(7):1405--1456, 2020.

\bibitem{LangNeron_rationalpoints}
S.~Lang and A.~N\'{e}ron.
\newblock Rational points of abelian varieties over function fields.
\newblock {\em Amer. J. Math.}, 81:95--118, 1959.

\bibitem{Lazarsfeld_PositivityI}
R.~Lazarsfeld.
\newblock {\em Positivity in algebraic geometry. {I}}, volume~48 of {\em Ergebnisse der Mathematik und ihrer Grenzgebiete. 3. Folge. A Series of Modern Surveys in Mathematics [Results in Mathematics and Related Areas. 3rd Series. A Series of Modern Surveys in Mathematics]}.
\newblock Springer-Verlag, Berlin, 2004.
\newblock Classical setting: line bundles and linear series.

\bibitem{MasserZannier_torsionEC0}
D.~Masser and U.~Zannier.
\newblock Torsion anomalous points and families of elliptic curves.
\newblock {\em C. R. Math. Acad. Sci. Paris}, 346(9-10):491--494, 2008.

\bibitem{Masser_Specialization}
D.~W. Masser.
\newblock Specializations of endomorphism rings of abelian varieties.
\newblock {\em Bull. Soc. Math. France}, 124(3):457--476, 1996.

\bibitem{Milne_AV}
J.~S. Milne.
\newblock Abelian varieties.
\newblock In {\em Arithmetic geometry ({S}torrs, {C}onn., 1984)}, pages 103--150. Springer, New York, 1986.

\bibitem{Mumford_GIT}
D.~Mumford, J.~Fogarty, and F.~Kirwan.
\newblock {\em Geometric invariant theory}, volume~34 of {\em Ergebnisse der Mathematik und ihrer Grenzgebiete (2) [Results in Mathematics and Related Areas (2)]}.
\newblock Springer-Verlag, Berlin, third edition, 1994.

\bibitem{PS_Analytic09}
Y.~Peterzil and S.~Starchenko.
\newblock Complex analytic geometry and analytic-geometric categories.
\newblock {\em J. Reine Angew. Math.}, 626:39--74, 2009.

\bibitem{PilaZannier_MM}
J.~Pila and U.~Zannier.
\newblock Rational points in periodic analytic sets and the {M}anin-{M}umford conjecture.
\newblock {\em Atti Accad. Naz. Lincei Rend. Lincei Mat. Appl.}, 19(2):149--162, 2008.

\bibitem{Pink_comb_published}
R.~Pink.
\newblock A combination of the conjectures of {M}ordell-{L}ang and {A}ndr\'{e}-{O}ort.
\newblock In {\em Geometric methods in algebra and number theory}, volume 235 of {\em Progr. Math.}, pages 251--282. Birkh\"{a}user Boston, Boston, MA, 2005.

\bibitem{Pink_comb_unpublished}
R.~Pink.
\newblock A common generalization of the conjectures of {A}ndr\'e-{O}ort, {M}anin-{M}umford, and {M}ordell-{L}ang.
\newblock {A}vailable at \url{https://people.math.ethz.ch/~pink/ftp/AOMMML.pdf}, 2005.

\bibitem{Raynaud_BookFA}
M.~Raynaud.
\newblock {\em Faisceaux amples sur les sch\'{e}mas en groupes et les espaces homog\`enes}.
\newblock Lecture Notes in Mathematics, Vol. 119. Springer-Verlag, Berlin-New York, 1970.

\bibitem{Remond_IntersectionII}
G.~R\'{e}mond.
\newblock Intersection de sous-groupes et de sous-vari\'{e}t\'{e}s. {II}.
\newblock {\em J. Inst. Math. Jussieu}, 6(2):317--348, 2007.

\bibitem{Remond_IntersectionIII}
G.~R\'{e}mond.
\newblock Intersection de sous-groupes et de sous-vari\'{e}t\'{e}s. {III}.
\newblock {\em Comment. Math. Helv.}, 84(4):835--863, 2009.

\bibitem{Silverberg_fod}
A.~Silverberg.
\newblock Fields of definition for homomorphisms of abelian varieties.
\newblock {\em J. Pure Appl. Algebra}, 77(3):253--262, 1992.

\bibitem{Silverman_Specialization}
J.~H. Silverman.
\newblock Heights and the specialization map for families of abelian varieties.
\newblock {\em J. Reine Angew. Math.}, 342:197--211, 1983.

\bibitem{UY_special}
E.~Ullmo and A.~Yafaev.
\newblock A characterization of special subvarieties.
\newblock {\em Mathematika}, 57(2):263--273, 2011.

\bibitem{Viada_optimal}
E.~Viada.
\newblock The optimality of the bounded height conjecture.
\newblock {\em J. Th\'{e}or. Nombres Bordeaux}, 21(3):769--784, 2009.

\bibitem{YuanZhang_ALB}
X.~Yuan and S.-W. Zhang.
\newblock Adelic line bundles on quasi-projective varieties.
\newblock {\em Available on arXiv}, 2024.

\bibitem{Zhang_ICM98}
S.-W. Zhang.
\newblock Small points and {A}rakelov theory.
\newblock In {\em Proceedings of the {I}nternational {C}ongress of {M}athematicians, {V}ol. {II} ({B}erlin, 1998)}, number Extra Vol. II, pages 217--225, 1998.

\bibitem{Zilber_exp}
B.~Zilber.
\newblock Exponential sums equations and the {S}chanuel conjecture.
\newblock {\em J. London Math. Soc. (2)}, 65(1):27--44, 2002.

\end{thebibliography}
